\documentclass[11pt,reqno]{amsart}
\usepackage[margin=3cm]{geometry}
\usepackage{color}               
\usepackage{faktor}
\usepackage{tikz}
\usepackage{tikz-cd}
\usepackage{spverbatim}
\usepackage{float}
\usepackage[pdftex,hyperfootnotes=false,colorlinks=false,hypertexnames=false]{hyperref}

\usepackage{amsmath,amssymb,anyfontsize,enumerate,stmaryrd,mathtools, thmtools,tikz,txfonts,nccmath}
\usepackage[oldstyle]{libertine}%
\usepackage[T1]{fontenc}
\usepackage[final]{microtype}
\usepackage[utf8]{inputenc}
\usepackage{csquotes}
\makeatletter
\renewcommand*\libertine@figurestyle{LF}
\makeatother
\usepackage[libertine,libaltvw,liby]{newtxmath}
\makeatletter
\renewcommand*\libertine@figurestyle{OsF}
\makeatother
\usepackage[noerroretextools,backend=biber,style=alphabetic,maxalphanames=5,giveninits,sorting=nyt,%
maxbibnames=99,doi=false,isbn=false,url=false,eprint=false]{biblatex}
\usepackage{tikz-cd}
\usepackage[noabbrev,capitalise]{cleveref}
\usepackage{autonum}
\theoremstyle{plain}
    \newtheorem{theorem}{Theorem}
    \newtheorem{construction/theorem}[theorem]{Construction/Theorem}
    \newtheorem{corollary}[theorem]{Corollary}
    \newtheorem{lemma}[theorem]{Lemma}
    \newtheorem{proposition}[theorem]{Proposition}
\theoremstyle{definition}
    \newtheorem{remark}[theorem]{Remark}
    \newtheorem{example}[theorem]{Example}
    \newtheorem{definition}[theorem]{Definition}

\DeclareMathOperator{\Aut}{Aut}
\DeclareMathOperator{\trop}{trop}

\DeclareMathOperator {\coef}{coef}
\DeclareMathOperator {\ft}{ft}

\title[Tropical twisted Hurwitz numbers for elliptic curves]{Tropical twisted Hurwitz numbers for elliptic curves}
\author[M.~A.~Hahn]{Marvin Anas Hahn}
\address{M.~A.~Hahn: School of Mathematics 17, Westland Row, Trinity College Dublin, Dublin 2, Ireland}
\email{hahnma@tcd.ie}
\author[H.~Markwig]{Hannah Markwig}
\address{H.~Markwig: Universität Tübingen, Fachbereich Mathematik, Auf der Morgenstelle 10, 72076 Tübingen, Germany}
\email{hannah@math.uni-tuebingen.de}


\begin{document}

\maketitle
\begin{abstract}
Hurwitz numbers enumerate branched morphisms between Riemann surfaces. For a fixed elliptic target, Hurwitz numbers are intimately related to mirror symmetry following work of Dijkgraaf. In recent work of Chapuy and Dołega a new variant of Hurwitz numbers with fixed genus $0$ target was introduced that includes maps between between non-orientiable surfaces. These numbers are called $b$-Hurwitz numbers and are polynomials in a parameter $b$ which measures the non-orientability of the involved maps. An interpretation in terms of factorisations of $b$-Hurwitz numbers for $b=1$, so-called twisted Hurwitz numbers, was found in work of Burman and Fesler. In previous work, the authors derived a tropical geometry interpretation of these numbers. In this paper, we introduce a natural generalisation of twisted Hurwitz numbers with elliptic targets within the framework of symmetric groups. We derive a tropical interpretation of these invariants, relate them to Feynman integrals and derive an expression as a matrix element of an operator in the bosonic Fock space.
\end{abstract}

\section{Introduction}
Hurwitz numbers count branched covers of Riemann surfaces with fixed numerical data. They originate from Hurwitz' orginal work in \cite{hurwitz1892algebraische} and have developed to important invariants in enumerative geometry. 
There are various equivalent definitions of Hurwitz numbers arising from different fields of mathematics. The one most important for this work is its interpretation via monodromy representations as an enumeration of factorisations in the symmetric group. As elliptic curves are the simplest cases of Calabi-Yau varieties, Hurwitz numbers of elliptic curves play a role in mirror symmetry. Dijkgraaf studied the relation between generating functions of Hurwitz numbers of an elliptic curve and Feynman intergrals \cite{Dij95}.

In recent work of Chapuy and Dołega \cite{chapuy2020non} a new class of Hurwitz numbers was introduced, called $b$-Hurwitz numbers depending on a parameter $b$. For $b=0$ one obtains classical Hurwitz numbers, while for $b=1$ these invariants specialise to an enumeration of covers between possibly non-orientable surface. Following \cite{BF21}, this enumeration gives rise to \textit{twisted Hurwitz numbers} which were proved to admit a definition in terms of counting factorisations in the symmetric group in \textit{loc. cit.} that mirrors its classical counterpart. In previous work \cite{HM22}, the authors developed a tropical geometry framework for the study of twisted Hurwitz numbers. So far, twisted Hurwitz numbers (and $b$-Hurwitz numbers) have only been studied for the enumeration of maps with a genus $0$ target. In this work, we introduce twisted Hurwitz numbers of an elliptic curve, also in terms of analogous factorisations in the symmetric group. Motivated by the fact that tropical geometry also provides a natural framework for the study of covers of an elliptic curve \cite{BBBM13,hahn2022triply}, we also study the tropical geometry of our twisted Hurwitz numbers of an elliptic curve.

\subsection{Elliptic Hurwitz numbers}
We first introduce the class of Hurwitz numbers of an elliptic curve showing up in the mirror symmetry relation involving Feynman integrals.

\begin{definition}[Hurwitz numbers of an elliptic curve]\label{def-hur}
Let $E$ be an elliptic curve (i.e.\ a Riemann surface of genus $1$), $g\ge1$ a non-negative integer, and $d>0$ a positive integer. Moreover, we fix $p_1,\dots,p_{2g-2}\in E$. Then, we consider covers of degree $d$, $f\colon S\to E$, such that
\begin{itemize}
    \item $S$ is a Riemann surface of genus $g$,
    \item the ramification profile of $p_1,\dots,p_{2g-2}$ is $(2,1\dots,1)$.
\end{itemize}
Two covers $f\colon S\to E$ and $f'\colon S'\to E$ are called equivalent if there exists a homeomorphism $h\colon S\to S'$, such that $f= f'\circ h$.\\
Then, we define the \textbf{Hurwitz number of the elliptic curve} $E$ as
\begin{equation}
    h_{d,g}=\sum_{[f]}\frac{1}{|\mathrm{Aut}(f)|},
\end{equation}
where the sum runs over all equivalence classes of covers as above.
\end{definition}
Such Hurwitz numbers are in fact topological invariants, i.e.\ they do not depend on the algebraic structure of the Riemann surfaces.
Via monodromy representations (see e.g. \cite{CM16}), Hurwitz numbers of an elliptic curve can be computed in terms of factorisations in the symmetric group.

\begin{lemma}\label{lem-monodrom}
The Hurwitz number $h_{d,g}$ equals $\frac{1}{d!}$ times the number of tuples
$$(\sigma,\tau_1,\ldots,\tau_{2g-2},\alpha)\in (\mathbb{S}_d)^{2g}$$
that satisfy the following:
\begin{enumerate}
    \item each $\tau_i$ is a transposition,
    \item the product of these permutations satisfies the following equation:$$\tau_{2g-2}\ldots\tau_{1}\sigma  =\alpha \sigma \alpha^{-1},$$
    \item the subgroup 
    $$\langle \sigma,\tau_1,\ldots,\tau_{2g-2},\alpha\rangle$$ acts transitively on the set $\{1,\ldots,d\}$.
\end{enumerate}

\end{lemma}

The idea for the proof of Lemma \ref{lem-monodrom} is to lift loops in the fundamental group of the elliptic curve to paths in the covering surface, see Figure \ref{fig:fundamentalgroup}.

\begin{figure}
    \centering
    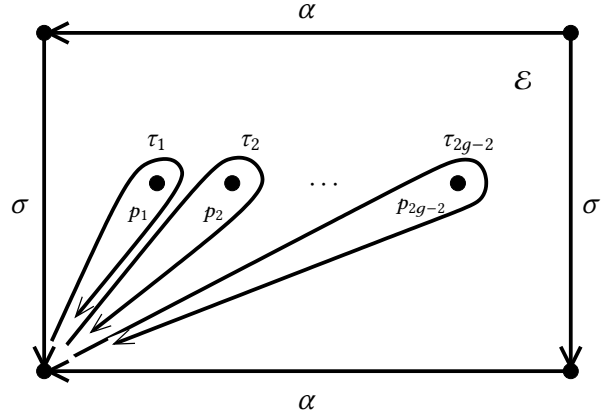
    \caption{A sketch of a cut open elliptic curve and the paths in its fundamental group which create the tuples to be counted to obtain a Hurwitz number via monodromy representations.}
    \label{fig:fundamentalgroup}
\end{figure}

\subsection{Twisted elliptic Hurwitz numbers}

We fix the involution $$\tau= (1\;\; d+1) (2 \;\;d+2) \ldots (d \;\;2d)\in \mathbb{S}_{2d}$$ and use the notation

\begin{equation}
    B_d=C(\tau)=\{\sigma \in \mathbb{S}_{2d}\;|\; \sigma \tau \sigma^{-1}=\tau\},\quad C^\sim(\tau)=\{\sigma \in \mathbb{S}_{2d}\;|\; \tau \sigma \tau^{-1} = \tau \sigma \tau =  \sigma^{-1}\}.
\end{equation}

We further define the subset $B^\sim_d \subset C^\sim(\tau)$ consisting of those permutations that have no self-symmetric cycles (see \cite[Lemma 2.1]{BF21}). 
We define the \textbf{twisted Hurwitz numbers of an elliptic curve} in terms of the symmetric group.

\begin{definition}[Twisted Hurwitz numbers of an elliptic curve]
\label{def:twisthur}
Fix a genus $g$ and a degree $d$.
Then the twisted Hurwitz number $\tilde{h}_{d,g}$ of degree $d$ and genus $g$ of an elliptic curve is defined to be
$\frac{1}{(2d)!!}$ times the number of tuples 

$$(\sigma,\eta_1\ldots,\eta_{g-1},\alpha)\in \mathbb({S}_{2d})^{g+1}$$
that satisfy the following conditions:

\begin{enumerate}
    \item each $\eta_s$ is a transposition, $\eta_s=(i_s\;\;j_s)$ such that $j_s\neq \tau(i_s)$,
    \item $\sigma \in B^\sim_d$,
    \item $\alpha \in B_d$,
    \item the product of these permutations satisfies the following equation:
    $$\eta_1\ldots\eta_{g-1}\sigma (\tau\eta_{g-1}\tau)\ldots (\tau\eta_1\tau) =\alpha \sigma \alpha^{-1},$$
    \item the subgroup 
    $$\langle \sigma,\eta_1,\ldots,\eta_{g-1}, (\tau\eta_{g-1}\tau),\ldots ,(\tau\eta_1\tau),\alpha\rangle$$ acts transitively on the set $\{1,\ldots,2d\}$.
\end{enumerate}


\end{definition}

The motivation to call these counts of factorisations in the symmetric group twisted Hurwitz numbers of an elliptic curve is coming from Lemma \ref{lem-monodrom}.
We can also drop the transitivity condition. On the tropical side, this corresponds to allowing disconnected source curves for our covers. We denote these numbers by $\tilde{h}_{d,g}^\bullet$.

\begin{example}\label{ex-twistednumber}
  The twisted Hurwitz number $\tilde{h}_{2,3}$ equals $16$.
\end{example}

\begin{remark}
    It is tempting to expect that $\tilde{h}_{d,g}$ admits a geometric interpretation similar to twisted Hurwitz numbers with genus $0$ target (see \cite[Section 2.2]{chapuy2020non} and the discussion after Remark 6 in \cite{HM22}). A reasonable expectation could be that elliptic twisted Hurwitz numbers count maps to an elliptic target $E$ with an orientation reversing involution that respects an orientation reversing involution on $E$. We leave the question of such a geometric interpretation as an open problem.
\end{remark}

\subsection{Main Results.}
In \cref{sec-troptwist}, we develop a tropical approach to elliptic twisted Hurwitz numbers. We introduce tropical elliptic twisted Hurwitz numbers as enumerations of tropical coverings of a tropical elliptic curve, i.e. as certain maps between graphs. As the main result of this section, in \cref{thm-corres} we prove a correspondence theorem stating that elliptic twisted Hurwitz numbers and their tropical counterparts coincide. This tropical interpretation then allows us to derive an expression in \cref{thm-feynman} of elliptic twisted Hurwitz numbers in terms of Feynman diagrams in \cref{sec-feynman}. We note that the computation of elliptic Hurwitz numbers in terms of Feynman diagrams was first derived in \cite{Dij95} and a new proof employing tropical techniques was given in \cite{BBBM13}. Finally, we follow the slogan \textit{bosonification is tropicalisation} which has now been established in a plethora of works \cite{block2016refined,CJMR16,cavalieri2021counting,hahn2022tropical,hahn2020wall} to express elliptic twisted Hurwitz numbers as matrix element on the bosonic Fock space in \cref{sec-fockspace}.

\subsection*{Acknowledgements} The authors thanks Raphaël Fesler for many useful discussions. The second author acknowledges support by the Deutsche Forschungsgemeinschaft (DFG, German Research Foundation), Project-ID 286237555, TRR 195. Computations have been made using the Computer Algebra System \textsc{gap} and \textsc{OSCAR} \cite{GAP4, Oscar}. We thank an anonymous referee for useful comments.

\section{Twisted tropical covers of an elliptic curve}
\label{sec-troptwist}

In \cite{HM22}, we have introduced twisted versions of tropical Hurwitz numbers. We now generalize to consider twisted versions of tropical covers of an elliptic curve and their counts, building on \cite{CJM10, BBM10, BBBM13}.
We start by recalling the basic notions of tropical curves and covers. Then we introduce twisted tropical covers of an elliptic curve, which can roughly be viewed as tropical covers with an involution. By fixing branch points, we produce a finite count of twisted tropical covers of an elliptic curve for which we show in the following that it coincides with the corresponding twisted Hurwitz number of an elliptic curve. Readers with a background in the theory of tropical curves are pointed to the fact that we only consider explicit tropical curves in the following, i.e.\ there is no genus hidden at vertices.

\begin{definition}[Abstract tropical curves]
An abstract tropical curve is a graph $\Gamma$ with the following data:
\begin{enumerate}
    \item The vertex set of $\Gamma$ is denoted by $V(\Gamma)$ and the edge set of $\Gamma$ is denoted by $E(\Gamma)$.
    \item The $1$-valent vertices of $\Gamma$ are called \textit{leaves} and the edges adjacent to leaves are called \textit{ends}.
    \item The set of edges $E(\Gamma)$ is partitioned into the set of ends $E^\infty(\Gamma)$ and the set of \textit{internal edges} $E^0(\Gamma)$.
    \item There is a length function
    \begin{equation}
        \ell\colon E(\Gamma)\to\mathbb{R}\cup\{\infty\},
    \end{equation}
    such that $\ell^{-1}(\infty)=E^\infty(\Gamma)$.
\end{enumerate}
The genus of an abstract tropical curve $\Gamma$ is defined as the first Betti number of the underlying graph. An isomorphism of abstract tropical curves is an isomorphism of the underlying graphs that respects the length function. The combinatorial type of an abstract tropical curve is the underlying graph without the length function.
\end{definition}

\begin{definition}
    A tropical elliptic curve $E$ is a circle of a given length. It may have several two-valent vertices. In the following we will refer to any tropical elliptic curve as $E$ and do not specify the number of vertices when it is clear from the context.
\end{definition}

\begin{example}\label{ex-abstractcurves}
    Figure \ref{fig:exabstractcurves} shows two abstract tropical curves of genus $3$. We have not specified edge lengths in the picture.
    \begin{figure}
        \centering

\tikzset{every picture/.style={line width=0.75pt}} 

\begin{tikzpicture}[x=0.75pt,y=0.75pt,yscale=-1,xscale=1]

\draw   (100.2,147.47) .. controls (100.2,136.42) and (115.87,127.47) .. (135.2,127.47) .. controls (154.53,127.47) and (170.2,136.42) .. (170.2,147.47) .. controls (170.2,158.51) and (154.53,167.47) .. (135.2,167.47) .. controls (115.87,167.47) and (100.2,158.51) .. (100.2,147.47) -- cycle ;
\draw   (191,149.67) .. controls (191,138.62) and (206.67,129.67) .. (226,129.67) .. controls (245.33,129.67) and (261,138.62) .. (261,149.67) .. controls (261,160.71) and (245.33,169.67) .. (226,169.67) .. controls (206.67,169.67) and (191,160.71) .. (191,149.67) -- cycle ;
\draw    (211.07,131.33) -- (211,167.67) ;
\draw    (241.47,131.67) -- (241.4,168) ;
\draw    (133.67,167.4) .. controls (118.33,156.73) and (115.07,139.33) .. (135.2,127.47) ;
\draw    (133.67,167.4) .. controls (149.67,159) and (151.73,141.87) .. (135.2,127.47) ;

\end{tikzpicture}

         \caption{Two abstract tropical curves of genus $3$. Edge lengths are not specified in the picture.}
        \label{fig:exabstractcurves}
    \end{figure}
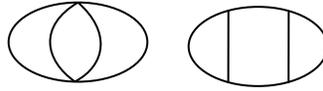
\end{example}

Next, we define the notion of tropical covers. We restrict to the case where the target is either a tropical elliptic curve $E$ or a subdivided version of $\mathbb{R}$, i.e.\ a line with some $2$-valent vertices.

\begin{definition}[Tropical covers]
Let the target $\Gamma_2$ be either a tropical elliptic curve $E$ or a subdivided version of $\mathbb{R}$.
A tropical cover between abstract tropical curves $\pi\colon \Gamma_1\to\Gamma_2$ is a surjective harmonic map, i.e.:
\begin{enumerate}
    \item We have $\pi(V(\Gamma_1))\subset V(\Gamma_2)$.
    \item Let $e\in E(\Gamma_1)$. Then, we interpret $e$ and $\pi(e)$ as intervals $[0,\ell(e)]$ and $[0,\ell(\pi(e))]$ respectively. We require $\pi$ restricted to $e$ to be a bijective integer linear function $[0,\ell(e)]\to[0,\ell(\pi(e))]$ given by $t\mapsto \omega(e)\cdot t$, with $\omega(e) \in \mathbb{Z}$. If $\pi(e)\in V(\Gamma_2)$, we define $\omega(e)=0$. We call $\omega(e)$ the \textit{weight} of $e$.
    \item For a vertex $v\in V(\Gamma_1)$, we denote by $\mathrm{Inc}(v)$ the set of incoming edges at $v$ (edges adjacent to $v$ mapping to the left of $\pi(v)$) and by $\mathrm{Out}(v)$ the set of outgoing edges at $v$ (edges adjacent to $v$ mapping to the right of $\pi(v)$). We then require
    \begin{equation}
            \sum_{e\in\mathrm{Inc}(v)}\omega(e)=\sum_{e\in\mathrm{Out}(v)}\omega(e).
    \end{equation}
    This number is called the local degree of $\pi$ at $v$.
    We call this equality the \textit{harmonicity} or \textit{balancing condition}.
    For a point $v$ in the interior of an edge $e$ of $\Gamma_1$, the local degree of $\pi$ at $v$ is defined to be the weight $\omega(e)$.
\end{enumerate}
Moreover, we define the \textit{degree} of $\pi$ as the half of the sum of local degrees at all vertices and internal points of $\Gamma_1$ in the preimage of a given vertex of $\Gamma_2$. By the harmonicity condition, the degree is independent of the choice of vertex of $\Gamma_2$.\\
For any end $e$ of $\Gamma_2$, we define a partition $\mu_e$ as the partition of weights of ends of $\Gamma_1$ mapping to $e$. We call $\mu_e$ the \textit{ramification profile} of $e$.\\
We call two tropical covers $\pi_1\colon\Gamma_1\to\Gamma_2$ and $\pi_2\colon\Gamma_1'\to\Gamma_2$ equivalent if there exists an isomorphism $g\colon\Gamma_1\to\Gamma_1'$ of metric graphs, such that $\pi_2\circ g=\pi_1$.
\end{definition}

We are now ready to give a definition of twisted tropical covers, which may be viewed as tropical covers admitting an involution with specified locus.

\begin{definition}[Twisted tropical covers of $E$]
We define a twisted topical cover of $E$ to be a tropical cover $\pi\colon\Gamma_1\to E$ with an involution $\iota\colon\Gamma_1\to\Gamma_1$ which respects the cover $\pi$, such that:
\begin{itemize}
    \item we have $g-1$ branch points $p_1,\dots,p_{g-1}\in E$ which we set as vertices,
    \item in the preimage of each branch point $p_i$, there are either two $3$-valent vertices or one $4$-valent vertex,
    \item the edges adjacent to a $4$-valent vertex all have the same weight,
    \item the fixed locus of $\iota$ is exactly the set of $4$-valent vertices.
\end{itemize}
\end{definition}

\begin{example}\label{ex-twotwistedcovers}
    Using the two source curves from Example \ref{ex-abstractcurves} (see Figure \ref{fig:exabstractcurves}), we can build twisted tropical covers of degree $2$, see Figure \ref{fig:twotwistedcovers}.

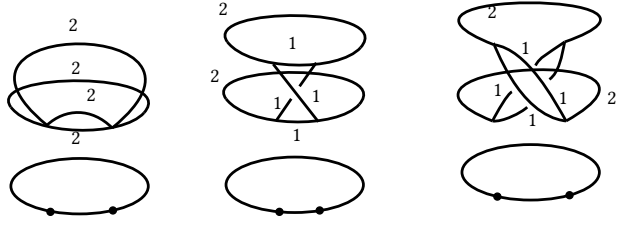
\begin{figure}
    \centering

\scalebox{0.7}{

\tikzset{every picture/.style={line width=0.75pt}} 

\begin{tikzpicture}[x=0.75pt,y=0.75pt,yscale=-1,xscale=1]

\draw  [line width=1.5]  (100,146) .. controls (100,134.95) and (122.16,126) .. (149.5,126) .. controls (176.84,126) and (199,134.95) .. (199,146) .. controls (199,157.05) and (176.84,166) .. (149.5,166) .. controls (122.16,166) and (100,157.05) .. (100,146) -- cycle ;
\draw [line width=1.5]    (100,89) .. controls (116,110) and (180,111) .. (196,93) ;
\draw [line width=1.5]    (100,89) .. controls (81,63) and (222,63) .. (196,93) ;
\draw [line width=1.5]    (126,102) .. controls (139,90) and (158,89) .. (173,104) ;
\draw [line width=1.5]    (126,102) .. controls (29,23) and (271,24) .. (173,104) ;
\draw  [line width=1.5]  (255,145) .. controls (255,133.95) and (277.16,125) .. (304.5,125) .. controls (331.84,125) and (354,133.95) .. (354,145) .. controls (354,156.05) and (331.84,165) .. (304.5,165) .. controls (277.16,165) and (255,156.05) .. (255,145) -- cycle ;
\draw [line width=1.5]    (255,83) .. controls (271,104) and (335,105) .. (351,87) ;
\draw [line width=1.5]    (255,83) .. controls (236,57) and (377,57) .. (351,87) ;
\draw  [line width=1.5]  (425,136) .. controls (425,124.95) and (447.16,116) .. (474.5,116) .. controls (501.84,116) and (524,124.95) .. (524,136) .. controls (524,147.05) and (501.84,156) .. (474.5,156) .. controls (447.16,156) and (425,147.05) .. (425,136) -- cycle ;
\draw [line width=1.5]    (447,99) .. controls (344,58) and (601,42) .. (500,99) ;
\draw [line width=1.5]    (448,44) .. controls (343,3) and (605,5) .. (499,42) ;
\draw [line width=1.5]    (256,42) .. controls (272,63) and (336,64) .. (352,46) ;
\draw [line width=1.5]    (256,42) .. controls (237,16) and (378,16) .. (352,46) ;
\draw [line width=1.5]    (289,58) -- (321.5,99.5) ;
\draw [line width=1.5]    (302,83) -- (291,99) ;
\draw [line width=1.5]    (448,44) .. controls (472,47) and (492,85) .. (500,99) ;
\draw [line width=1.5]    (448,44) .. controls (448,41) and (465,95) .. (500,99) ;
\draw [line width=1.5]    (320,57) -- (308,74) ;
\draw [line width=1.5]    (447,99) .. controls (450,93) and (458,79) .. (461,78) ;
\draw [line width=1.5]    (478,59) .. controls (481,53) and (486,51) .. (499,42) ;
\draw [line width=1.5]    (447,99) .. controls (454,100) and (468,93) .. (472,89) ;
\draw [line width=1.5]    (488,69) .. controls (495,64) and (497,52) .. (499,42) ;
\draw  [fill={rgb, 255:red, 0; green, 0; blue, 0 }  ,fill opacity=1 ] (126,164.5) .. controls (126,163.12) and (127.12,162) .. (128.5,162) .. controls (129.88,162) and (131,163.12) .. (131,164.5) .. controls (131,165.88) and (129.88,167) .. (128.5,167) .. controls (127.12,167) and (126,165.88) .. (126,164.5) -- cycle ;
\draw  [fill={rgb, 255:red, 0; green, 0; blue, 0 }  ,fill opacity=1 ] (171,163.5) .. controls (171,162.12) and (172.12,161) .. (173.5,161) .. controls (174.88,161) and (176,162.12) .. (176,163.5) .. controls (176,164.88) and (174.88,166) .. (173.5,166) .. controls (172.12,166) and (171,164.88) .. (171,163.5) -- cycle ;
\draw  [fill={rgb, 255:red, 0; green, 0; blue, 0 }  ,fill opacity=1 ] (291,164.5) .. controls (291,163.12) and (292.12,162) .. (293.5,162) .. controls (294.88,162) and (296,163.12) .. (296,164.5) .. controls (296,165.88) and (294.88,167) .. (293.5,167) .. controls (292.12,167) and (291,165.88) .. (291,164.5) -- cycle ;
\draw  [fill={rgb, 255:red, 0; green, 0; blue, 0 }  ,fill opacity=1 ] (320,163.5) .. controls (320,162.12) and (321.12,161) .. (322.5,161) .. controls (323.88,161) and (325,162.12) .. (325,163.5) .. controls (325,164.88) and (323.88,166) .. (322.5,166) .. controls (321.12,166) and (320,164.88) .. (320,163.5) -- cycle ;
\draw  [fill={rgb, 255:red, 0; green, 0; blue, 0 }  ,fill opacity=1 ] (448,153.5) .. controls (448,152.12) and (449.12,151) .. (450.5,151) .. controls (451.88,151) and (453,152.12) .. (453,153.5) .. controls (453,154.88) and (451.88,156) .. (450.5,156) .. controls (449.12,156) and (448,154.88) .. (448,153.5) -- cycle ;
\draw  [fill={rgb, 255:red, 0; green, 0; blue, 0 }  ,fill opacity=1 ] (500,152.5) .. controls (500,151.12) and (501.12,150) .. (502.5,150) .. controls (503.88,150) and (505,151.12) .. (505,152.5) .. controls (505,153.88) and (503.88,155) .. (502.5,155) .. controls (501.12,155) and (500,153.88) .. (500,152.5) -- cycle ;

\draw (140,23.4) node [anchor=north west][inner sep=0.75pt]    {$2$};
\draw (142,54.4) node [anchor=north west][inner sep=0.75pt]    {$2$};
\draw (142,105.4) node [anchor=north west][inner sep=0.75pt]    {$2$};
\draw (153,74.4) node [anchor=north west][inner sep=0.75pt]    {$2$};
\draw (248,12.4) node [anchor=north west][inner sep=0.75pt]    {$2$};
\draw (242,60.4) node [anchor=north west][inner sep=0.75pt]    {$2$};
\draw (442,14.4) node [anchor=north west][inner sep=0.75pt]    {$2$};
\draw (301.13,102.78) node [anchor=north west][inner sep=0.75pt]    {$1$};
\draw (315.13,74.78) node [anchor=north west][inner sep=0.75pt]    {$1$};
\draw (287.13,79.78) node [anchor=north west][inner sep=0.75pt]    {$1$};
\draw (528,76.4) node [anchor=north west][inner sep=0.75pt]    {$2$};
\draw (471,92.4) node [anchor=north west][inner sep=0.75pt]    {$1$};
\draw (466,40.4) node [anchor=north west][inner sep=0.75pt]    {$1$};
\draw (446,70.4) node [anchor=north west][inner sep=0.75pt]    {$1$};
\draw (493,76.4) node [anchor=north west][inner sep=0.75pt]    {$1$};
\draw (298.13,36.78) node [anchor=north west][inner sep=0.75pt]    {$1$};

\end{tikzpicture}}

    \caption{{\color{blue}Three } twisted tropical covers of $E$ of degree $2$ and genus $3$. The labels on the edges denote the edge weight. Edges which are not labeled have weight one. The involution $\iota$ is supposed to exchange respective edges on top resp.\ bottom of the picture.}
    \label{fig:twotwistedcovers}
\end{figure}
    
\end{example}

As usually for Hurwitz numbers, our enumeration of twisted tropical covers will take automorphisms into account. We give the following definition which specifies automorphisms that take the involution into account.

\begin{definition}[Automorphisms]
    Let $\pi:\Gamma\rightarrow E$ be a twisted tropical cover with involution $\iota:\Gamma\rightarrow \Gamma$. An automorphism of $\pi$ is a morphism of abstract tropical curves (i.e.\ a map of metric graphs) $f:\Gamma\rightarrow\Gamma$ respecting the cover and the involution, i.e.\ $\pi\circ f = \pi$ and $f\circ \iota = \iota \circ f$. We denote the group of automorphisms of $\pi$ by $\Aut(\pi)$.
\end{definition}

\begin{example}\label{ex-automorphisms}
    The twisted tropical cover depicted in Figure \ref{fig:twotwistedcovers} on the left has an automorphism group of size 4:  we can independently exchange the two pairs of weight $2$ edges mapping to the same segment of $E$. The twisted tropical cover  in the middle has an automorphism group of size $2$ (generated by the involution): we can exchange the two edges of weight $2$, with the two pairs of edges of weight $1$ following along. The one on the right has an automorphism group of size $4$: we can exchange two parallel edges of weight $1$ in addition to the involution.
\end{example}

\begin{definition}[Quotient graph $\Gamma/\iota$, see \cite{HM22}] \label{def-quotient}
Let $\pi\colon \Gamma\to E$ be a twisted tropical cover with involution $\iota\colon\Gamma\to\Gamma$. The involution $\iota$ induces a symmetric relation on the vertex and edge sets of $\Gamma$: We define for $v,v'\in V(\Gamma)$  (resp. $e,e'\in E(\Gamma)$) that $v\sim v'$ (resp. $e\sim e'$) if and only $\iota(v)=v'$ (resp. $\iota(e)=e'$). We define $\Gamma/\iota$ as the graph with vertex set $V(\Gamma)/\sim$ and edge set $E(\Gamma)/\sim$ with natural identifications. For $e=[e',e'']\in E(\Gamma/\iota)$ we define the length $\ell(e)$ as $\ell(e')=\ell(e'')$ and its weight $\omega(e)$ with respect to $\pi$ to be the weight $\omega(e')=\omega(e'')$. In this way, we obtain a tropical cover from the quotient graph $\Gamma/\iota$ to $E$, which has $2$-valent vertices coming from the $4$-valent vertices of $\Gamma$, and $3$-valent vertices else.
\end{definition}

\begin{example}\label{ex-quotients}
    The quotient covers of the three twisted covers from Figure \ref{fig:twotwistedcovers} are depicted in Figure \ref{fig:quotientcover}. The middle and right tropical cover have the same quotient graph.

    \begin{figure}
        \centering

\tikzset{every picture/.style={line width=0.75pt}} 

\begin{tikzpicture}[x=0.75pt,y=0.75pt,yscale=-1,xscale=1]

\draw   (117.87,270.72) .. controls (117.87,262.57) and (133.54,255.97) .. (152.87,255.97) .. controls (172.2,255.97) and (187.87,262.57) .. (187.87,270.72) .. controls (187.87,278.86) and (172.2,285.47) .. (152.87,285.47) .. controls (133.54,285.47) and (117.87,278.86) .. (117.87,270.72) -- cycle ;
\draw   (216.53,271.72) .. controls (216.53,263.57) and (232.2,256.97) .. (251.53,256.97) .. controls (270.86,256.97) and (286.53,263.57) .. (286.53,271.72) .. controls (286.53,279.86) and (270.86,286.47) .. (251.53,286.47) .. controls (232.2,286.47) and (216.53,279.86) .. (216.53,271.72) -- cycle ;
\draw  [fill={rgb, 255:red, 0; green, 0; blue, 0 }  ,fill opacity=1 ] (266.67,284.88) .. controls (266.67,284.1) and (267.3,283.47) .. (268.08,283.47) .. controls (268.87,283.47) and (269.5,284.1) .. (269.5,284.88) .. controls (269.5,285.67) and (268.87,286.3) .. (268.08,286.3) .. controls (267.3,286.3) and (266.67,285.67) .. (266.67,284.88) -- cycle ;
\draw  [fill={rgb, 255:red, 0; green, 0; blue, 0 }  ,fill opacity=1 ] (235.67,285.22) .. controls (235.67,284.43) and (236.3,283.8) .. (237.08,283.8) .. controls (237.87,283.8) and (238.5,284.43) .. (238.5,285.22) .. controls (238.5,286) and (237.87,286.63) .. (237.08,286.63) .. controls (236.3,286.63) and (235.67,286) .. (235.67,285.22) -- cycle ;
\draw  [fill={rgb, 255:red, 0; green, 0; blue, 0 }  ,fill opacity=1 ] (164.33,284.22) .. controls (164.33,283.43) and (164.97,282.8) .. (165.75,282.8) .. controls (166.53,282.8) and (167.17,283.43) .. (167.17,284.22) .. controls (167.17,285) and (166.53,285.63) .. (165.75,285.63) .. controls (164.97,285.63) and (164.33,285) .. (164.33,284.22) -- cycle ;
\draw  [fill={rgb, 255:red, 0; green, 0; blue, 0 }  ,fill opacity=1 ] (132,283.22) .. controls (132,282.43) and (132.63,281.8) .. (133.42,281.8) .. controls (134.2,281.8) and (134.83,282.43) .. (134.83,283.22) .. controls (134.83,284) and (134.2,284.63) .. (133.42,284.63) .. controls (132.63,284.63) and (132,284) .. (132,283.22) -- cycle ;
\draw    (150,218.3) -- (150.78,244.97) ;
\draw [shift={(150.83,246.97)}, rotate = 268.33] [color={rgb, 255:red, 0; green, 0; blue, 0 }  ][line width=0.75]    (10.93,-3.29) .. controls (6.95,-1.4) and (3.31,-0.3) .. (0,0) .. controls (3.31,0.3) and (6.95,1.4) .. (10.93,3.29)   ;
\draw    (249.67,218.3) -- (250.44,244.97) ;
\draw [shift={(250.5,246.97)}, rotate = 268.33] [color={rgb, 255:red, 0; green, 0; blue, 0 }  ][line width=0.75]    (10.93,-3.29) .. controls (6.95,-1.4) and (3.31,-0.3) .. (0,0) .. controls (3.31,0.3) and (6.95,1.4) .. (10.93,3.29)   ;
\draw   (216.2,191.05) .. controls (216.2,182.9) and (231.87,176.3) .. (251.2,176.3) .. controls (270.53,176.3) and (286.2,182.9) .. (286.2,191.05) .. controls (286.2,199.2) and (270.53,205.8) .. (251.2,205.8) .. controls (231.87,205.8) and (216.2,199.2) .. (216.2,191.05) -- cycle ;
\draw   (117.53,189.05) .. controls (117.53,180.9) and (133.2,174.3) .. (152.53,174.3) .. controls (171.86,174.3) and (187.53,180.9) .. (187.53,189.05) .. controls (187.53,197.2) and (171.86,203.8) .. (152.53,203.8) .. controls (133.2,203.8) and (117.53,197.2) .. (117.53,189.05) -- cycle ;
\draw  [fill={rgb, 255:red, 0; green, 0; blue, 0 }  ,fill opacity=1 ] (165.33,202.55) .. controls (165.33,201.77) and (165.97,201.13) .. (166.75,201.13) .. controls (167.53,201.13) and (168.17,201.77) .. (168.17,202.55) .. controls (168.17,203.33) and (167.53,203.97) .. (166.75,203.97) .. controls (165.97,203.97) and (165.33,203.33) .. (165.33,202.55) -- cycle ;
\draw  [fill={rgb, 255:red, 0; green, 0; blue, 0 }  ,fill opacity=1 ] (133,201.55) .. controls (133,200.77) and (133.63,200.13) .. (134.42,200.13) .. controls (135.2,200.13) and (135.83,200.77) .. (135.83,201.55) .. controls (135.83,202.33) and (135.2,202.97) .. (134.42,202.97) .. controls (133.63,202.97) and (133,202.33) .. (133,201.55) -- cycle ;
\draw    (236.54,204.39) .. controls (243.71,199.22) and (257,197.39) .. (267,204.39) ;

\draw (150.5,165.63) node [anchor=north west][inner sep=0.75pt]  [font=\tiny] [align=left] {$\displaystyle 2$};
\draw (147.83,194.63) node [anchor=north west][inner sep=0.75pt]  [font=\tiny] [align=left] {$\displaystyle 2$};
\draw (211.5,196.3) node [anchor=north west][inner sep=0.75pt]  [font=\tiny] [align=left] {$\displaystyle 2$};

\end{tikzpicture}

\caption{The quotient covers of the {three} twisted covers from Example \ref{ex-twotwistedcovers}, see Figure \ref{fig:twotwistedcovers}.}
        \label{fig:quotientcover}
    \end{figure}
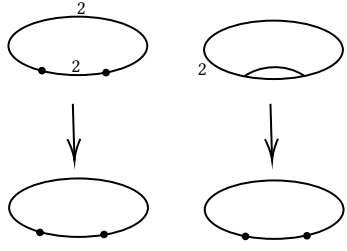
\end{example}

\begin{proposition}\label{prop-automquotient}
Let  $\overline{\pi}:\overline{\Gamma}\rightarrow E$ be a (connected) quotient of a tropical twisted cover (see Definition \ref{def-quotient}). Assume $\overline{\Gamma}$ has $c$ $2$-valent vertices and is of genus $g'$.
Then

$$ \sum_{\pi} \frac{1}{\sharp \Aut(\pi)}= \frac{2^{g'}-\delta_{0c}}{2^{c+1}\cdot \sharp \Aut(\overline{\pi})},$$
where the sum goes over all (connected) twisted tropical covers  $\pi:\Gamma\rightarrow E$  with involution $\iota$ whose quotient $\Gamma/\iota\rightarrow E$ equals $\overline{\pi}:\overline{\Gamma}\rightarrow E$ and $\delta_{0c}=1$ if $c=0$ and $0$ else.

\end{proposition}

\begin{proof}
Assume momentarily we had a rational graph $\overline{\Gamma}$ without any $2$-valent vertices and we would like to understand its preimages under taking a quotient with respect to an involution $\iota$. That means, we take two copies of every edge of $\overline{\Gamma}$ such that the involution exchanges the edges in the pair. When drawing a picture, we like to draw one edge on the top and one on the bottom. Now let us consider an adjacent vertex, and another edge starting from this vertex. Again, we take two copies, but since the involution exchanges the two it does not matter which we draw on top and which we draw on the bottom. Thus, for the whole rational graph, we have a unique preimage under taking the quotient which just consists of two disjoint copies of $\overline{\Gamma}$.

Next, let us consider a graph $\overline{\Gamma}$ of genus $g'>0$, but momentarily still without any $2$-valent vertices.
Pick $g'$ edges $e_1,\ldots,e_{g'}$ whose complement in $\overline{\Gamma}$ form a tree. To understand the preimage of taking the quotient, we can start by taking the preimage of the tree and putting in two copies for each of the missing edges $e_1,\ldots,e_{g'}$. As discussed above, the preimage of the tree just consists of two copies which we call the top and the bottom part. When we insert a pair of edges for $e_1$, we now have two options: we can either let one connect top with top, and the other bottom with bottom, or both can connect top with bottom. We have the same choice for all $g'$ edges, yielding $2^{g'}$ choices for preimages under taking the quotient. However, one of these (where we always connect top with top and bottom with bottom) is disconnected and should therefore be discarded. Also, not all the $2^{g'}$ choices have to be distinct, and this happens in the presence of automorphisms of $\overline{\pi}$: the latter arise due to parallel edges which are mapped in the same way (in particular, with the same weight). Assume $e_1$ is one of two parallel edges and $e_2$ is an edge of a cycle that is mapped to the cycle of the elliptic curve. If we require the two preimages of $e_1$ to connect top with bottom, then we can get the same picture for both choices for the two preimages of $e_2$, see Figure \ref{fig-autom}. Dividing by the size of the automorphism group of $\overline{\pi}$, we even out such overcountings however. 

Finally, let us consider what happens in the presence of $2$-valent vertices. First, every choice of preimage will be connected, so we do not have to subtract one when counting possibilities. Second, for every $4$-valent vertex in each preimage $\pi$, we can exchange the edges of one adjacent twisted pair, leading to an extra automorphism of $\pi$ which descends to the identity on the quotient $\overline{\pi}$.

The extra factor of $2$ in the denominator arises because of the involution, which yields an extra contribution to the automorphism group of each preimage.

\end{proof}

\begin{remark}
    We note that \cref{prop-automquotient} generalises to covers of $\mathbb{R}$, and disconnected tropical twisted covers as well. Indeed, let $\overline{\pi}\colon \overline{\Gamma}\to \mathbb{R}$ be a quotient of a tropical twisted cover with $c$ $2$-valent vertices and $r$ connected components, each of genus $g_i$. Thus, the genus of $\overline{\Gamma}$ is $g'=\sum g_i-r+1$. As in the proof of \cref{prop-automquotient}, we obtain $2^{\sum g_i}$ many preimages under taking the quotient. Now, we aim to count automorphisms. The argument is the same as in the proof above with the difference, that the involutions acts indepedently on each component and thus, we obtain a factor in the denominator of $2^r$. Thus, in total, we obtain
    \begin{equation}
        \sum_{\pi}\frac{1}{\sharp\mathrm{Aut}(\pi)}=\frac{2^{\sum g_i}}{2^r2^c\sharp\mathrm{Aut}(\overline{\pi})},
    \end{equation}
    where the sum now runs over possibly disconnected twisted covers $\pi$ with quotient $\overline{\pi}$. To conclude, we observe however that $\sum g_i=g'+r+1$ and thus, we obtain
        \begin{equation}
        \sum_{\pi}\frac{1}{\sharp\mathrm{Aut}(\pi)}=\frac{2^{g'}}{2^{c+1}\sharp\mathrm{Aut}(\overline{\pi})}.
    \end{equation}
    The only difference to the connected case in \cref{prop-automquotient} is the factor $\delta_{0,c}$ that ensured connectedness which obviously does not play a role here.
\end{remark}

\begin{figure}
    \centering

\tikzset{every picture/.style={line width=0.75pt}} 

\begin{tikzpicture}[x=0.75pt,y=0.75pt,yscale=-1,xscale=1]

\draw    (51.5,206.89) .. controls (64.5,199.86) and (77.17,203.19) .. (81.5,206.89) ;
\draw    (51.5,206.89) .. controls (62.17,212.19) and (70.5,213.19) .. (81.5,206.89) ;
\draw    (51.5,206.89) .. controls (-50.25,179.84) and (182.5,185.84) .. (81.5,206.89) ;
\draw    (216.75,152.64) .. controls (229.75,145.61) and (242.42,148.94) .. (246.75,152.64) ;
\draw    (216.75,152.64) .. controls (227.42,157.94) and (235.75,158.94) .. (246.75,152.64) ;
\draw    (216.75,152.64) .. controls (115,125.59) and (347.75,131.59) .. (246.75,152.64) ;
\draw    (217.25,199.14) .. controls (230.25,192.11) and (242.92,195.44) .. (247.25,199.14) ;
\draw    (217.25,199.14) .. controls (227.92,204.44) and (236.25,205.44) .. (247.25,199.14) ;
\draw    (217.25,199.14) .. controls (115.5,172.09) and (348.25,178.09) .. (247.25,199.14) ;
\draw    (335.5,153.64) .. controls (348.5,146.61) and (361.17,149.94) .. (365.5,153.64) ;
\draw    (335.5,153.64) .. controls (344.63,161.59) and (352.63,189.09) .. (366,200.14) ;
\draw    (335.5,153.64) .. controls (233.75,126.59) and (466.5,132.59) .. (365.5,153.64) ;
\draw    (336,200.14) .. controls (344.63,192.09) and (361.63,158.84) .. (365.5,153.64) ;
\draw    (336,200.14) .. controls (346.67,205.44) and (355,206.44) .. (366,200.14) ;
\draw    (336,200.14) .. controls (234.25,173.09) and (467,179.09) .. (366,200.14) ;
\draw    (214.38,252.89) .. controls (227.38,245.86) and (240.04,249.19) .. (244.38,252.89) ;
\draw    (214.38,252.89) .. controls (225.04,258.19) and (233.38,259.19) .. (244.38,252.89) ;
\draw    (214.38,252.89) .. controls (187.25,219.84) and (347,263.34) .. (244.88,299.39) ;
\draw    (214.88,299.39) .. controls (227.88,292.36) and (240.54,295.69) .. (244.88,299.39) ;
\draw    (214.88,299.39) .. controls (225.54,304.69) and (233.88,305.69) .. (244.88,299.39) ;
\draw    (214.88,299.39) .. controls (121.25,212.09) and (345.38,231.84) .. (244.38,252.89) ;
\draw    (333.88,255.59) .. controls (344.13,269.09) and (360.04,298.39) .. (364.38,302.09) ;
\draw    (333.88,255.59) .. controls (344.63,250.59) and (354.88,252.34) .. (363.88,255.59) ;
\draw    (333.88,255.59) .. controls (308.63,216.09) and (465.38,281.04) .. (364.38,302.09) ;
\draw    (334.38,302.09) .. controls (343.38,292.34) and (360.63,256.09) .. (363.88,255.59) ;
\draw    (334.38,302.09) .. controls (345.04,307.39) and (353.38,308.39) .. (364.38,302.09) ;
\draw    (334.38,302.09) .. controls (233.13,202.84) and (464.88,234.54) .. (363.88,255.59) ;

\draw (35.42,183.75) node  [font=\footnotesize] [align=left] {\begin{minipage}[lt]{13.94pt}\setlength\topsep{0pt}
$\displaystyle 2$
\end{minipage}};
\draw (73,218.5) node  [font=\footnotesize] [align=left] {\begin{minipage}[lt]{13.94pt}\setlength\topsep{0pt}
$\displaystyle 1$
\end{minipage}};
\draw (73,196.75) node  [font=\footnotesize] [align=left] {\begin{minipage}[lt]{13.94pt}\setlength\topsep{0pt}
$\displaystyle 1$
\end{minipage}};
\draw (198.42,130.5) node  [font=\footnotesize] [align=left] {\begin{minipage}[lt]{13.94pt}\setlength\topsep{0pt}
$\displaystyle 2$
\end{minipage}};
\draw (237,164.25) node  [font=\footnotesize] [align=left] {\begin{minipage}[lt]{13.94pt}\setlength\topsep{0pt}
$\displaystyle 1$
\end{minipage}};
\draw (235.5,142.25) node  [font=\footnotesize] [align=left] {\begin{minipage}[lt]{13.94pt}\setlength\topsep{0pt}
$\displaystyle 1$
\end{minipage}};
\draw (199.17,175.75) node  [font=\footnotesize] [align=left] {\begin{minipage}[lt]{13.94pt}\setlength\topsep{0pt}
$\displaystyle 2$
\end{minipage}};
\draw (234,190.25) node  [font=\footnotesize] [align=left] {\begin{minipage}[lt]{13.94pt}\setlength\topsep{0pt}
$\displaystyle 1$
\end{minipage}};
\draw (238.25,211.64) node  [font=\footnotesize] [align=left] {\begin{minipage}[lt]{13.94pt}\setlength\topsep{0pt}
$\displaystyle 1$
\end{minipage}};
\draw (321.67,130) node  [font=\footnotesize] [align=left] {\begin{minipage}[lt]{13.94pt}\setlength\topsep{0pt}
$\displaystyle 2$
\end{minipage}};
\draw (340.75,192) node  [font=\footnotesize] [align=left] {\begin{minipage}[lt]{13.94pt}\setlength\topsep{0pt}
$\displaystyle 1$
\end{minipage}};
\draw (356,159.5) node  [font=\footnotesize] [align=left] {\begin{minipage}[lt]{13.94pt}\setlength\topsep{0pt}
$\displaystyle 1$
\end{minipage}};
\draw (320.17,176.75) node  [font=\footnotesize] [align=left] {\begin{minipage}[lt]{13.94pt}\setlength\topsep{0pt}
$\displaystyle 2$
\end{minipage}};
\draw (370,191.5) node  [font=\footnotesize] [align=left] {\begin{minipage}[lt]{13.94pt}\setlength\topsep{0pt}
$\displaystyle 1$
\end{minipage}};
\draw (354.75,212.39) node  [font=\footnotesize] [align=left] {\begin{minipage}[lt]{13.94pt}\setlength\topsep{0pt}
$\displaystyle 1$
\end{minipage}};
\draw (189.29,257.75) node  [font=\footnotesize] [align=left] {\begin{minipage}[lt]{13.94pt}\setlength\topsep{0pt}
$\displaystyle 2$
\end{minipage}};
\draw (235.38,288.75) node  [font=\footnotesize] [align=left] {\begin{minipage}[lt]{13.94pt}\setlength\topsep{0pt}
$\displaystyle 1$
\end{minipage}};
\draw (234.63,311.89) node  [font=\footnotesize] [align=left] {\begin{minipage}[lt]{13.94pt}\setlength\topsep{0pt}
$\displaystyle 1$
\end{minipage}};
\draw (320.79,235.45) node  [font=\footnotesize] [align=left] {\begin{minipage}[lt]{13.94pt}\setlength\topsep{0pt}
$\displaystyle 2$
\end{minipage}};
\draw (352.13,263.2) node  [font=\footnotesize] [align=left] {\begin{minipage}[lt]{13.94pt}\setlength\topsep{0pt}
$\displaystyle 1$
\end{minipage}};
\draw (343.13,289.2) node  [font=\footnotesize] [align=left] {\begin{minipage}[lt]{13.94pt}\setlength\topsep{0pt}
$\displaystyle 1$
\end{minipage}};
\draw (393.79,280.2) node  [font=\footnotesize] [align=left] {\begin{minipage}[lt]{13.94pt}\setlength\topsep{0pt}
$\displaystyle 2$
\end{minipage}};
\draw (371.38,293.7) node  [font=\footnotesize] [align=left] {\begin{minipage}[lt]{13.94pt}\setlength\topsep{0pt}
$\displaystyle 1$
\end{minipage}};
\draw (354.63,314.84) node  [font=\footnotesize] [align=left] {\begin{minipage}[lt]{13.94pt}\setlength\topsep{0pt}
$\displaystyle 1$
\end{minipage}};
\draw (267.04,281.5) node  [font=\footnotesize] [align=left] {\begin{minipage}[lt]{13.94pt}\setlength\topsep{0pt}
$\displaystyle 2$
\end{minipage}};

\end{tikzpicture}

    \caption{On the left, a quotient cover $\overline{\pi}$ of genus $2$. On the right, the $4$ choices of preimage under taking the quotient, as in the proof of Proposition \ref{prop-automquotient}. Because of the automorphism of the quotient cover, the upper right and the lower right choice get identified. The top middle is disconnected and should be discarded. The lower middle has an automorphism group of size $4$ due to the automorphism of $\overline{\pi}$. Altogether, we have $\frac{1}{2}+\frac{1}{4}=\frac{4-1}{2\cdot 2}$ preimages counted with one over the size of their automorphism group, as predicted by Proposition \ref{prop-automquotient}.}
    \label{fig-autom}
\end{figure}
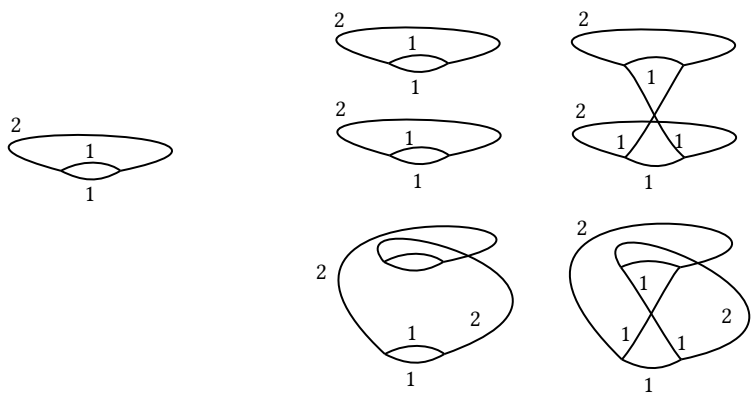


We are now ready to define twisted tropical Hurwitz numbers of an elliptic curve.

\begin{definition}[Twisted tropical Hurwitz number of an elliptic curve]
\label{def:twisttrophn}
We define the tropical twisted  Hurwitz 
number of an elliptic curve $\tilde{h}^{\trop}_{d,g}$ to be  the weighted enumeration of equivalence classes of twisted tropical covers of degree $d$ and genus $g$ of a tropical elliptic curve $E$, such that each equivalence class $[\pi\colon \Gamma\to E]$ is counted with multiplicity
\begin{equation}
    2^{g-1}\cdot \frac{1}{|\Aut(\pi)|}\cdot\prod_V (\omega_V-1)\prod_e\omega(e),
\end{equation}
where the first product goes over all $4$-valent vertices and $\omega_V$ denotes the weight of the adjacent edges, while the second product is taken over all edges of the quotient graph $\Gamma/\iota$ and $\omega(e)$ denotes their weights.

\end{definition}

\begin{example}\label{ex-twistedtropicalnumber}
    The twisted tropical Hurwitz number $\tilde{h}^{\trop}_{2,3}$ equals $16$. There are five twisted tropical covers for $d=2$ and $g=3$. We may obtain all from the sketches in Figure \ref{fig:twotwistedcovers}. The middle and right depicted maps each give rise to two tropical covers depending on the labelling of the branch points, i.e. for the tropical elliptic curve $E$ on the bottom, either the left vertex is labelled $p_1$ and the right vertex is labelled $p_2$ or the left vertex is labelled $p_2$ and the right vertex is labelled $p_1$.
    The left cover in Figure \ref{fig:twotwistedcovers} has multiplicity $$2^2\cdot \frac{1}{4}\cdot (2-1)\cdot (2-1)\cdot 2\cdot 2=4.$$
    
    Each cover coming from the middle picture has multiplicity 
    $$2^2\cdot \frac{1}{2}\cdot 2=4.$$

Each cover coming from the right picture has multiplicity
$$2^2\cdot \frac{1}{4}\cdot 2=2.$$
    
    For the sizes of the automorphism groups, see Example \ref{ex-automorphisms}. Thus, we obtain $\tilde{h}^{\trop}_{2,3}=16$ in total. Note that this number coincides with the twisted Hurwitz number we computed in Example \ref{ex-twistednumber}.
\end{example}

\begin{remark}\label{rem-countquotients}
Notice that by Proposition \ref{prop-automquotient}, twisted tropical Hurwitz numbers of an elliptic curve can also be determined by counting quotient covers directly.
If $\Gamma$ has $c$ $4$-valent vertices and is of genus $g$, then by an Euler characteristics computation  the quotient $\overline{\Gamma}$ has genus $g'=\frac{1}{2}\cdot (g-c+1)$. For the multiplicities of the preimages of a quotient cover, only the factor $\frac{1}{|\Aut(\pi)|}$ differ, all others remain. But the sum of the $\frac{1}{|\Aut(\pi)|}$ is obtained via Proposition \ref{prop-automquotient}, and so a quotient cover $\overline{\pi}$ has to be counted with multiplicity
\begin{equation}
  \frac{2^{g'}-\delta_{0c}}{2^{c+1}}  \cdot 2^{g-1}\cdot \frac{1}{|\Aut(\overline{\pi})|}\cdot\prod_V (\omega_V-1)\prod_e\omega(e)=(2^{g'}-\delta_{0c})  \cdot 2^{2g'-3}\cdot \frac{1}{|\Aut(\overline{\pi})|}\cdot\prod_V (\omega_V-1)\prod_e\omega(e),
\end{equation}

where $c$ denotes the number of $2$-valent vertices of the source graph $\Gamma$, the first product goes over all $2$-valent vertices and $\omega_V$ denotes the weight of the adjacent edges, while the second product is taken over all edges and $\omega(e)$ denotes their weights.
    
\end{remark}

The equality observed in Example \ref{ex-twistedtropicalnumber} is no coincidence, as shown in the following Theorem:

\begin{theorem}[Correspondence Theorem]\label{thm-corres}
The twisted Hurwitz number of $E$ equals its tropical counterpart, i.e.\
$$\tilde{h}^{\trop}_{d,g}=\tilde{h}_{d,g}.$$
\end{theorem}
\begin{proof} Let $\pi:\Gamma\rightarrow E$ be a twisted tropical cover of degree $d$. We pick a base point $p_0$ between $p_{g-1}$ and $p_1$ and cut the elliptic curve open at $p_0$. We also cut the preimages of $p_0$ under $\pi$, thus obtaining a twisted tropical cover $\tilde{\pi}:\tilde{\Gamma}\rightarrow \mathbb{R}$. In the untwisted case, this is explained in detail in Construction 4.4 in \cite{BBBM13}. 

The twisted Hurwitz number counts tuples
\begin{align}
    \tilde{h}_{d,g} &= \frac{1}{(2d)!!} \sharp \{(\sigma,\eta_1\ldots,\eta_{g-1},\alpha)\} 
    \end{align}
    as in Definition \ref{def-hur}. We can split these tuples and first list tuples of the form $\{(\sigma,\eta_1\ldots,\eta_{g-1})\}$, combining each such tuple with a list of possible  $\alpha$.
    Each tuple $\{(\sigma,\eta_1\ldots,\eta_{g-1})\}$ yields a twisted tropical cover $\tilde{\pi}$ of $\mathbb{R}$ as in Construction/Theorem 14 \cite{HM22}. This cover has left and right ends of weights given by the cycle lengths of $\sigma$. 
    For each twisted tropical cover of $\mathbb{R}$ having ends of the same weights in both directions, the number of tuples of the form above leading to this cover equals its tropical multiplicity by the correspondence theorem (see Proposition 18 and Remark 6 in \cite{HM22}). The tropical multiplicity equals
    $$2^{g-1}\prod_V (\omega_V-1) \prod_e\omega (e) \prod_K \frac{1}{\omega_K}\cdot \frac{1}{|\Aut(\tilde{\pi})|}$$
    where the first product goes over all $4$-valent vertices $V$ and $\omega_V$ denotes the weight of its adjacent edges, the second product goes over all pairs of twisted internal edges of the source $\tilde{\Gamma}$ of the twisted tropical cover $\tilde{\pi}$, and the third over all twisted pairs of components $K$ which consist of a single edge of weight $\omega_K$. 
Combining tuples of the form $\{(\sigma,\eta_1\ldots,\eta_{g-1})\}$ with possible $\alpha$ amounts to gluing a twisted tropical cover of $\mathbb{R}$ to obtain a twisted tropical cover of $E$.
To determine the number of such gluings, we pass to the quotient covers on each side.
Given a twisted tropical cover $\pi$ of $E$ and its cut cover $\tilde{\pi}$, we consider the quotient cover $\overline{\pi}$ of $E$ and the quotient cut cover $\tilde{\overline{\pi}}$. Note that taking the quotient and cutting the cover commutes.

 When gluing a cut quotient cover of $\mathbb{R}$ to a quotient cover of $E$, we want to pair up left and right ends that should be glued. Each left end of the quotient cover corresponds to a pair of ends of the twisted cover. Assume the pair of cycles $c_1$, $\tau\circ c_1 \circ \tau$ corresponds to these two ends, and assume that our gluing merges the left end corresponding to $c_1$ with the right end corresponding to a cycle $c_2$ of the same length. We want to count the number of $\alpha$ that satisfy $c_2=\alpha\circ c_1\circ \alpha^{-1}$. Let $c_1=(c_{11},\ldots,c_{1\ell(c_1)})$, and $c_2=(c_{21},\ldots,c_{2\ell(c_1)})$. A choice of $\alpha$ is fixed by setting $\alpha(c_{11})=c_{2i}$ for any $i=1,\ldots,\ell(c_1)$. As we require that $\alpha \tau=\tau \alpha$ any element in the twisted cycle $\tau\circ c_1 \circ \tau$ of the form $\tau(c_{1j})$ must be mapped to $\tau(\alpha(c_{1j}))$ via $\alpha$. Thus, choosing a gluing on one of a pair of twisted ends of the cut cover fixes the gluing on the other. 

By the same argument as in \cite[Proposition 4.9]{BBBM13} the number of such $\alpha$ is given by
\begin{equation}\label{equ-npi}\prod_{e'}\omega({e'})^{c_{e'}}\cdot\frac{|\Aut(\tilde\pi)|}{|\Aut(\pi)|},\end{equation} 
where the product goes over all pairs of twisted edges $e'$ of $\Gamma$ that contain a preimage of the base
point $p_0$ of $E$ and $c_{e'}$ denotes the number of preimages in $e'$,
$c_{e'}= \#  (\pi^{-1}(p_0)\cap e') $.




We can group the tuples in the set according to the twisted tropical cover $\pi:C\rightarrow E$ they provide under the cut-and-join construction, see Construction/Theorem 14 in \cite{HM22}. Thus we can write $\tilde{h}_{d,g}$ as
$$
\tilde{h}_{d,g}=\frac{1}{(2d)!!}\cdot \sum_{\pi}\#\left\{(\sigma,\eta_1\ldots,\eta_{g-1},\alpha)\textnormal{ yielding the cover }\pi\right\}.$$
For a fixed
cover $\pi$, instead of counting tuples yielding $\pi$, we can count tuples $(\sigma,\eta_1\ldots,\eta_{g-1})$ yielding the cut twisted tropical cover
$\tilde\pi$ and then multiply with the number of appropriate $\alpha$, which we denote by
$n_{\tilde\pi,\pi}$:
$$
\tilde{h}_{d,g}=\frac{1}{(2d)!!}\cdot\sum_\pi
\#\left\{(\sigma,\eta_1\ldots,\eta_{g-1})\textnormal{ that provide the cover }\tilde{\pi}\right\}\cdot
n_{\tilde{\pi},\pi}.$$

By the above, the count of the
tuples yielding a cover $\tilde\pi$ divided by $(2d)!!$ equals $$2^{g-1}\frac{1}{|\Aut(\tilde{\pi})|}\cdot \prod_V(\omega_V-1)\cdot\prod_{\tilde
e} \omega({\tilde e})\cdot\prod_K
\frac{1}{\omega_K}$$ where the first product goes over the $4$-valent vertices, the second over all pairs of twisted internal edges $\tilde{e}$ of $\tilde{\Gamma}$ of weight $\omega({\tilde{e}})$ and the third over 
all twisted pairs of components $K$ consisting of a single edge of weight $\omega_K$. From the above, the number $n_{\tilde{\pi},\pi}$ can be
substituted by the expression in \cref{equ-npi}.

We obtain
$$
\tilde{h}_{d,g}=\sum_\pi
\frac{1}{|\Aut(\tilde{\pi})|}2^{g-1}\cdot\prod_V(\omega_V-1)\cdot \prod_{\tilde{e}}\omega({\tilde{e}})
\cdot\prod_K \frac{1}{\omega_K}\cdot\prod_ { e' }
\omega({e'})^{c_{e'}}\cdot\frac{|\Aut(\tilde\pi)|}{|\Aut(\pi)|}.
$$
A pair of twisted edges $e'$ of $\Gamma$ of weight $\omega({e'})$ having $c_{e'}$
preimages over the base point provides exactly $c_{e'}-1$ pairs of single-edge-components of weight
$\omega({e'})$  in the cut cover $\tilde\pi$. Vice versa, each such pair of
components comes from a pair of edges with multiple preimages over the
base point. Therefore the expression $\prod_K \frac{1}{\omega_K}\cdot\prod_ { e' }
\omega({e'})^{c_{e'}}$ simplifies to $\prod_{e'} \omega({e'})$. We obtain
$$
\tilde{h}_{d,g}=\sum_\pi
2^{g-1}\frac{1}{|\Aut(\pi)|}\cdot \prod_V(\omega_V-1)\cdot \prod_e \omega(e) = \tilde{h}_{d,g}^{\trop}$$
 and the theorem is proved.

\end{proof}

\begin{example}
    Consider the twisted tropical cover $\pi$ depicted in Figure \ref{fig:twotwistedcovers} in the middle with the vertex on the left labelled $p_1$ and the vertex on the right labelled $p_2$.
    Cutting it open, we obtain the twisted tropical cover $\tilde{\pi}$ of $\mathbb{R}$ depicted in Figure \ref{fig:cutopen}.

\begin{figure}
    \centering

\tikzset{every picture/.style={line width=0.75pt}} 

\begin{tikzpicture}[x=0.75pt,y=0.75pt,yscale=-1,xscale=1]

\draw  [fill={rgb, 255:red, 0; green, 0; blue, 0 }  ,fill opacity=1 ] (398.33,269.88) .. controls (398.33,269.1) and (398.97,268.47) .. (399.75,268.47) .. controls (400.53,268.47) and (401.17,269.1) .. (401.17,269.88) .. controls (401.17,270.67) and (400.53,271.3) .. (399.75,271.3) .. controls (398.97,271.3) and (398.33,270.67) .. (398.33,269.88) -- cycle ;
\draw  [fill={rgb, 255:red, 0; green, 0; blue, 0 }  ,fill opacity=1 ] (358.5,269.75) .. controls (358.5,268.97) and (359.13,268.33) .. (359.92,268.33) .. controls (360.7,268.33) and (361.33,268.97) .. (361.33,269.75) .. controls (361.33,270.53) and (360.7,271.17) .. (359.92,271.17) .. controls (359.13,271.17) and (358.5,270.53) .. (358.5,269.75) -- cycle ;
\draw    (379.17,221.33) -- (379.94,248) ;
\draw [shift={(380,250)}, rotate = 268.33] [color={rgb, 255:red, 0; green, 0; blue, 0 }  ][line width=0.75]    (10.93,-3.29) .. controls (6.95,-1.4) and (3.31,-0.3) .. (0,0) .. controls (3.31,0.3) and (6.95,1.4) .. (10.93,3.29)   ;
\draw    (320,270) -- (440,270) ;
\draw    (320,170) -- (360,170) ;
\draw    (320,200) -- (360,200) ;
\draw    (400,170) -- (440,170) ;
\draw    (400,200) -- (440,200) ;
\draw    (360,200) -- (400,200) ;
\draw    (360,170) -- (400,170) ;
\draw    (360,170) -- (400,200) ;
\draw    (400,170) -- (382.5,183.97) ;
\draw    (378.17,186.97) -- (360,200) ;

\draw (329,161) node [anchor=north west][inner sep=0.75pt]  [font=\tiny] [align=left] {$\displaystyle 2$};
\draw (329,191) node [anchor=north west][inner sep=0.75pt]  [font=\tiny] [align=left] {$\displaystyle 2$};
\draw (421,161) node [anchor=north west][inner sep=0.75pt]  [font=\tiny] [align=left] {$\displaystyle 2$};
\draw (421,191) node [anchor=north west][inner sep=0.75pt]  [font=\tiny] [align=left] {$\displaystyle 2$};

\end{tikzpicture}

     \caption{The twisted tropical cover on the right in Figure \ref{fig:twotwistedcovers} cut open at the back.}
    \label{fig:cutopen}
\end{figure}
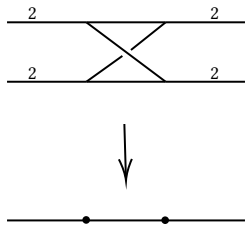

By the correspondence theorem for twisted tropical covers of $\mathbb{R}$ in \cite[Theorem 22]{HM22}, it accounts for $(2d)!!=8$ times its tropical multiplicity many tuples.
The cover has an automorphism group of size $2^2=4$, as we can independently exchange both pairs of twisted edges of weight $2$, with the edges of weight $1$ following along.
Its tropical multiplicity thus equals
$$ 2^2\cdot \frac{1}{4}=1.$$
This twisted tropical cover of $\mathbb{R}$ thus accounts for $8$ tuples of the form $(\sigma_1,\eta_1,\eta_2)$. By Lemma 16 \cite{HM22}, there are $2$ permutations suitable for $\sigma_1$, $(14)(23)$ and $(12)(34)$. Fix $\sigma_1=(14)(23)$ momentarily, the other choice is analogous. Then there are two choices for $\eta_1$, $(14)$ of $(23)$. For the next branch point, there are $2$ more choices for $\eta_2$, $(12)$ or $(34)$. Altogether, we obtain the $8$ tuples as expected.

The extra automorphism that the cut cover obtains (we have $\frac{|\Aut(\tilde{\pi})|}{|\Aut(\pi)|}=\frac{4}{2}=2$) allows to make an additional choice which left end should be glued to which right end.
Let us momentarily fix one of our $8$ tuples, $((14)(23), (14), (12))$. If we label all edges with the corresponding permutations, the two right ends are labeled with $(12)$ and $(34)$. Because of the extra automorphism, we can glue the left end labeled $(14)$ either to $(12)$ or to $(34)$. For each choice, we obtain as many $\alpha$ satisfying $\alpha\tau=\tau\alpha$ as the weight of one end, i.e.\ $2$. We thus obtain $4$ possible $\alpha$ to add to each of the $8$ tuples, yielding $32$ tuples of the form
$(\sigma_1,\eta_1,\eta_2,\alpha)$. 
For the tuple fixed above, the $4$ possible $\alpha$ we can add are
$$\{(24), (1234), (13), (1432)\}.$$
Dividing the $32$ tuples by $(2d)!!=8$, we expect the tropical multiplicity of the right cover of $E$ in Figure \ref{fig:twotwistedcovers} to be $4$. Indeed, in Example \ref{ex-twistedtropicalnumber} we already computed its tropical multiplicity to be
$$2^2\cdot \frac{1}{2}\cdot 2 = 4.$$

\end{example}

\begin{example}

We illustrate another example in \cref{fig:cuttingcovers8}. On the top left, we have the quotient cover $\overline{\pi}\colon\overline{\Gamma}\to E$ of an elliptic tropical twisted cover of degree $4$. Note that $\overline{\Gamma}$, as illustrated at the bottom of \cref{fig:cuttingcovers8} has three edges, two of weight $1$ and one of weight $2$. On the top right in this figure, we have the tropical twisted cover obtained by cutting $\overline{\pi}$ at $p_0$ and its preimages. In particular, we obtain two edges of weight $1$ arising from the same edge of $\overline{\Gamma}$. This is because this edge curls twice before reattaching again at the bottom to join to an edge of weight $2$.

\end{example}

 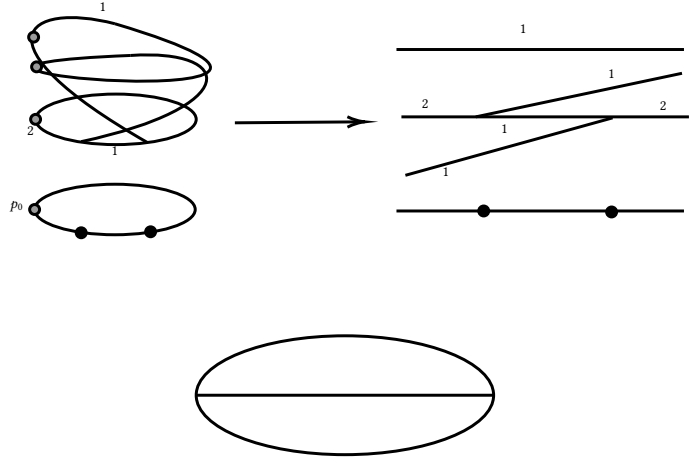
\begin{figure}

\scalebox{0.5}{

\tikzset{every picture/.style={line width=0.75pt}} 

\begin{tikzpicture}[x=0.75pt,y=0.75pt,yscale=-1,xscale=1,trim left=2.5cm]

\draw  [line width=2.25]  (183,235.5) .. controls (183,221.42) and (219.26,210) .. (264,210) .. controls (308.74,210) and (345,221.42) .. (345,235.5) .. controls (345,249.58) and (308.74,261) .. (264,261) .. controls (219.26,261) and (183,249.58) .. (183,235.5) -- cycle ;
\draw  [line width=2.25]  (184,144.5) .. controls (184,130.42) and (220.26,119) .. (265,119) .. controls (309.74,119) and (346,130.42) .. (346,144.5) .. controls (346,158.58) and (309.74,170) .. (265,170) .. controls (220.26,170) and (184,158.58) .. (184,144.5) -- cycle ;
\draw [line width=2.25]    (230,167) .. controls (393,132) and (386,73) .. (280,80) ;
\draw [line width=2.25]    (297,168) .. controls (118.5,66) and (175,25) .. (264,48) ;
\draw [line width=2.25]    (264,48) .. controls (593,146) and (-53,93) .. (280,80) ;
\draw  [fill={rgb, 255:red, 0; green, 0; blue, 0 }  ,fill opacity=1 ][line width=2.25]  (225,259) .. controls (225,256.24) and (227.24,254) .. (230,254) .. controls (232.76,254) and (235,256.24) .. (235,259) .. controls (235,261.76) and (232.76,264) .. (230,264) .. controls (227.24,264) and (225,261.76) .. (225,259) -- cycle ;
\draw  [fill={rgb, 255:red, 0; green, 0; blue, 0 }  ,fill opacity=1 ][line width=2.25]  (295,258) .. controls (295,255.24) and (297.24,253) .. (300,253) .. controls (302.76,253) and (305,255.24) .. (305,258) .. controls (305,260.76) and (302.76,263) .. (300,263) .. controls (297.24,263) and (295,260.76) .. (295,258) -- cycle ;
\draw  [fill={rgb, 255:red, 155; green, 155; blue, 155 }  ,fill opacity=1 ][line width=2.25]  (178,235.5) .. controls (178,232.74) and (180.24,230.5) .. (183,230.5) .. controls (185.76,230.5) and (188,232.74) .. (188,235.5) .. controls (188,238.26) and (185.76,240.5) .. (183,240.5) .. controls (180.24,240.5) and (178,238.26) .. (178,235.5) -- cycle ;
\draw  [fill={rgb, 255:red, 155; green, 155; blue, 155 }  ,fill opacity=1 ][line width=2.25]  (179,144.5) .. controls (179,141.74) and (181.24,139.5) .. (184,139.5) .. controls (186.76,139.5) and (189,141.74) .. (189,144.5) .. controls (189,147.26) and (186.76,149.5) .. (184,149.5) .. controls (181.24,149.5) and (179,147.26) .. (179,144.5) -- cycle ;
\draw  [fill={rgb, 255:red, 155; green, 155; blue, 155 }  ,fill opacity=1 ][line width=2.25]  (180,91.5) .. controls (180,88.74) and (182.24,86.5) .. (185,86.5) .. controls (187.76,86.5) and (190,88.74) .. (190,91.5) .. controls (190,94.26) and (187.76,96.5) .. (185,96.5) .. controls (182.24,96.5) and (180,94.26) .. (180,91.5) -- cycle ;
\draw  [fill={rgb, 255:red, 155; green, 155; blue, 155 }  ,fill opacity=1 ][line width=2.25]  (177,61.5) .. controls (177,58.74) and (179.24,56.5) .. (182,56.5) .. controls (184.76,56.5) and (187,58.74) .. (187,61.5) .. controls (187,64.26) and (184.76,66.5) .. (182,66.5) .. controls (179.24,66.5) and (177,64.26) .. (177,61.5) -- cycle ;
\draw [line width=2.25]    (548,74) -- (838,74) ;
\draw [line width=2.25]    (553,142) -- (843,142) ;
\draw [line width=2.25]    (628,142) -- (836,98) ;
\draw [line width=2.25]    (557,201) -- (766,143) ;
\draw [line width=2.25]    (548,237) -- (838,237) ;
\draw  [fill={rgb, 255:red, 0; green, 0; blue, 0 }  ,fill opacity=1 ][line width=2.25]  (631,237) .. controls (631,234.24) and (633.24,232) .. (636,232) .. controls (638.76,232) and (641,234.24) .. (641,237) .. controls (641,239.76) and (638.76,242) .. (636,242) .. controls (633.24,242) and (631,239.76) .. (631,237) -- cycle ;
\draw  [fill={rgb, 255:red, 0; green, 0; blue, 0 }  ,fill opacity=1 ][line width=2.25]  (760,238) .. controls (760,235.24) and (762.24,233) .. (765,233) .. controls (767.76,233) and (770,235.24) .. (770,238) .. controls (770,240.76) and (767.76,243) .. (765,243) .. controls (762.24,243) and (760,240.76) .. (760,238) -- cycle ;
\draw [line width=2.25]    (385,148) -- (512,147.03) ;
\draw [shift={(516,147)}, rotate = 179.56] [color={rgb, 255:red, 0; green, 0; blue, 0 }  ][line width=2.25]    (17.49,-5.26) .. controls (11.12,-2.23) and (5.29,-0.48) .. (0,0) .. controls (5.29,0.48) and (11.12,2.23) .. (17.49,5.26)   ;
\draw  [line width=2.25]  (346,423) .. controls (346,389.86) and (413.16,363) .. (496,363) .. controls (578.84,363) and (646,389.86) .. (646,423) .. controls (646,456.14) and (578.84,483) .. (496,483) .. controls (413.16,483) and (346,456.14) .. (346,423) -- cycle ;
\draw [line width=2.25]    (346,423) -- (646,423) ;

\draw (247,25) node [anchor=north west][inner sep=0.75pt]   [align=left] {$\displaystyle 1$};
\draw (174,150) node [anchor=north west][inner sep=0.75pt]   [align=left] {$\displaystyle 2$};
\draw (157,225) node [anchor=north west][inner sep=0.75pt]   [align=left] {$\displaystyle p_{0}$};
\draw (572,122.4) node [anchor=north west][inner sep=0.75pt]    {$2$};
\draw (593,191.4) node [anchor=north west][inner sep=0.75pt]    {$1$};
\draw (812,124.4) node [anchor=north west][inner sep=0.75pt]    {$2$};
\draw (671,46.4) node [anchor=north west][inner sep=0.75pt]    {$1$};
\draw (760,92.4) node [anchor=north west][inner sep=0.75pt]    {$1$};
\draw (652,147.4) node [anchor=north west][inner sep=0.75pt]    {$1$};
\draw (259,170) node [anchor=north west][inner sep=0.75pt]   [align=left] {$\displaystyle 1$};

\end{tikzpicture}

}
     \caption{On the top left, there is the quotient cover $\pi\colon\overline{\Gamma}\to E$ of an elliptic tropical twisted cover. Cutting the cover at $p_0$ and the preimages, we obtain the tropical twisted cover on the top right. At the bottom is the graph $\overline{\Gamma}$.}
     \label{fig:cuttingcovers8}
 \end{figure}

\section{Generating series in terms of Feynman integrals}
\label{sec-feynman}

In this section, we express elliptic twisted Hurwitz numbers as Feynman integrals. We assume that $g>2$ in the following. Consequently, in the quotient covers there cannot be loop edges.

In our context, the following definition of Feynman graph will be needed. These are exactly the graphs that appear as sources for quotients of twisted covers, up to  labeling.
\begin{definition}[Feynman graph]
    A Feynman  graph is a graph with  $2$- and $3$-valent vertices whose edges are labeled with $q_1,\ldots,q_r$ and whose vertices are labeled with $x_1,\ldots,x_s$.
\end{definition}

A Feynman integral depends on a Feynman graph and the choice of an order $\Omega$ of the vertices.

\begin{definition}[Edge Propagator]
Let $q_k$ be an edge of a Feynman graph, adjacent to two vertices $x_{k_1}$ and $x_{k_2}$, where we assume that $x_{k_1}<x_{k_2}$ in the order $\Omega$.

Given $w\in \mathbb{N}$, we define the coefficient $c_w$ of the  following propagator function to be
$$c_w:=\begin{cases}
    (w-1)\cdot w & \mbox{ if } x_{k_1} \mbox{ and } x_{k_2} \mbox{ are $2$-valent}\\
    \sqrt{w-1}\cdot w & \mbox{ if } x_{k_1} \mbox{ or } x_{k_2} \mbox{ is $2$-valent}\\
     w & \mbox{ if neither } x_{k_1} \mbox{ nor } x_{k_2} \mbox{ are $2$-valent}.\\
\end{cases} $$

We then define the propagator function of the edge $q_k$ to be 

$$P(q_k)=\sum_{w=1}^{\infty} c_w \Big(\frac{x_{k_1}}{x_{k_2}}\Big)^w+ \sum_{a_k=1}^\infty \Bigg(\sum_{w|a_k} c_w \Bigg(\Big(\frac{x_{k_1}}{x_{k_2}}\Big)^w+\Big(\frac{x_{k_2}}{x_{k_1}}\Big)^w\Bigg)\Bigg)q_k^{a_k}.$$

\end{definition}

\begin{definition}[Feynman integral]
Let $\Gamma$ be a Feynman graph and $\Omega$ be an order of its vertices. For each edge $q_k$, we denote its adjacent vertices by $x_{k_1}$ and $x_{k_2}$, where we assume that $x_{k_1}<x_{k_2}$ in the order $\Omega$.
We define the Feynman integral $I_{\Gamma,\Omega}(q_1,\ldots,q_r)$ to be
$$I_{\Gamma,\Omega}(q_1,\ldots,q_r)=\coef_{[x_1^0\ldots x_s^0]} \prod_{k=1}^r P(q_k).$$

Setting all $q_k$ equal to one variable $q$, we obtain the Feynman integral $I_{\Gamma,\Omega}(q)$.
    
\end{definition}


\begin{remark}
    Here, we consider Feynman integrals merely as formal power series. In the relation involving (usual) Hurwitz  numbers of an elliptic curve, the propagator series can, using a coordinate change, be transformed into a linear combination of the Weierstra\ss-$\wp$-function and an Eisenstein series. After this coordinate change, the Feynman integral can be viewed as a complex analytic path integral.
\end{remark}

\begin{remark}
Fix a genus $g>2$. 
A $3$-valent graph of genus $2$ has $2$ vertices, increasing the genus by one yields $2$ more vertices. It follows that a graph of genus $2g$ has $2g-2$ vertices if it is $3$
-valent. Every $4$-valent vertex can be viewed as a merging of $2$ $3$-valent vertices, thus a graph of genus $2g$ with $c$ $4$-valent vertices and only $3$-valent vertices else has $2g-2-c$ vertices.
It follows that the source of a twisted tropical cover of $E$ has $2g-2-c$ vertices, where $c$ denotes the number of $4$-valent vertices. That is, $2g-2-2c$ vertices are $3$-valent and $c$ are $4$-valent. When passing to the quotient cover, its source graph has $g-1-c$ many $3$-valent vertices and $c$ many $4$-valent vertices. Its total number of vertices is thus $g-1$, independent of the number of $4$-valent vertices in the twisted cover.
\end{remark}

\begin{definition}[Labeled quotient covers]
A \emph{labeled quotient cover} is a quotient of a twisted tropical cover for which the vertices and edges of its source are labeled like a Feynman graph.

    Fix a base point $p_0$ of a tropical elliptic curve $E$. Fix a genus $g>2$, and $g-1$ branch points $p_1,\ldots, p_{g-1}$ in $E$.

Given a labeled quotient cover $\overline{\pi}$, we can define its \emph{multidegree} $a\in \mathbb{N}^r$ to be the tuple whose $k$-th entry equals the sum of the weights of the preimages of the base point $p_0$ in the edge $q_k$. 
    Fix an order $\Omega$ on $g-1$ elements $x_1,\ldots,x_{g-1}$. Let $\Gamma$ be a Feynman graph.

 We define $\tilde{h}_{\Gamma,\Omega,a}$ to be the weighted number of labeled quotient covers whose source is of combinatorial type $\Gamma$, whose multidegree equals $a$ and such that the order given by the preimages of the branch points $\overline{\pi}^{-1}(p_1)<\ldots < \overline{\pi}^{-1}(p_{g-1})$ equals $\Omega$. Each such cover is weighted by $\frac{2^{g'}-\delta_{0c}}{2^{c+1}}  \cdot 2^{g-1}\cdot \prod_V (\omega_V-1)\cdot \omega(e).$ Since there are no nontrivial automorphisms in the presence of labels, this equals the multiplicity given in Remark \ref{rem-countquotients}.
\end{definition}

\begin{proposition}\label{prop-hGammaOmega} Let $\Gamma$ be a Feynman graph of genus $g>2$ and $c$ be its number of $2$-valent vertices. Let $a$ be a multidegree and $\Omega$ an order. The count of labeled quotient covers equals a coefficient of a Feynman integral:
 $$ \tilde{h}_{\Gamma,\Omega,a}=\frac{2^{g'}-\delta_{0c}}{2^{c+1}  \cdot 2^{g-1}}
 \cdot \coef_{[q_1^{a_1}\ldots  q_r^{a_r}]} I_{\Gamma,\Omega}(q_1,\ldots,q_r).$$  
\end{proposition}
 The proof follows ideas of \cite{BBBM13, BGM22}.

\begin{proof}
Expanding the  product $\prod_{k=1}^r P(q_k)$, the summands are equal to products of the form
$$\prod_{k=1}^r c_{w_k} \Big(\frac{x_{i}}{x_{j}}\Big)^{w_k} \cdot q_k^{a_k}.$$ If $a_k$ is zero, $w_k$ can be any element in $\mathbb{N}$, and $i=k_1$, $j=k_2$. If $a_k>0$, $w_k|a_k$ and $i$ can be either $k_1$ or $k_2$, and $j$ the remaining.
To each such summand, we associate a labeled quotient cover in the following way: We start by fixing as preimages of the branch points the vertices $x_i$ as imposed by the order $\Omega$. 

For a factor $c_{w_k} \Big(\frac{x_{i}}{x_{j}}\Big)^{w_k} \cdot q_k^{a_k}$ with $a_k=0$, we draw an edge labeled $q_k$ which goes from the vertex $x_{k_1}$ to $x_{k_2}$ without crossing over the base point. This is possible since the $x_i$ respect the order $\Omega$. We fix the weight of our edge to be $w_k$.

For a factor $c_{w_k} \Big(\frac{x_{i}}{x_{j}}\Big)^{w_k} \cdot q_k^{a_k}$ with $a_k>0$, we draw an edge labeled $q_k$: If $i=k_1$ we let it start at $x_{k_1}$ and connect with $x_{k_2}$ (where we think of the edges of our cover as oriented in the way imposed by the order $\Omega$). If $i=k_2$, we let it start at $x_{k_2}$ at connect it with $x_{k_1}$. We ''curl'' this edge in such a way that is passes $\frac{a_k}{w_k}$ times over the base point $p_0$. The weight of the edge in each case is defined to be $w_k$. 

We claim that in this way, we produce a labeled quotient cover contributing to $\tilde{h}_{\Gamma,\Omega,a}$ with $a=(a_1,\dots,a_r)$. Since we used the edge $q_k$ to connect its neighboring vertices in $\Gamma$, the source of the covers is of combinatorial type $\Gamma$ by construction. The multidegree is $a$, since for each $k$ with $a_k=0$, we let our edge not pass over the base point, whereas for each $k$ with $a_k>0$ the edge of weight $w_k$ passes $\frac{a_k}{w_k}$ times over the base point, leading to the entry $a_k$ in the multidegree. The order $\Omega$ is also respected by construction.

What remains to be seen is that we obtained indeed a cover, i.e.\ the balancing condition has to be satisfied. This holds true since a product as above only contributes to the Feynman integral $I_{\Gamma,\Omega}(q_1,\ldots,q_r)$ if its total degree in the $x_i$ vanishes. The total power of $x_i$ equals, by construction, the signed sum of the weights of its adjacent edges. The fact that the degree in $x_i$ is zero is thus equivalent to the balancing condition at vertex $x_i$.

In this way, we obtain a bijection between summands contributing to the Feynman integral and labeled quotient covers. What about multiplicities? In the Feynman integral, a summand contributes $\prod_k c_{w_k}$. Thus, the summand contributes  $\frac{2^{g'}-\delta_{0c}}{2^{c+1}}  \cdot 2^{g-1}\cdot \prod_k c_{w_k}$ to the right hand side. We have to show that this equals the multiplicity with which the labeled quotient cover is counted in Remark \ref{rem-countquotients}, i.e.\ that $\prod_k c_{w_k}= \prod_V(\omega_V-1)\cdot \prod_e\omega_e$.
Recall that the weight of the edge $q_k$ equals $w_k$. 
For a $2$-valent vertex $V$, we have a factor of $\omega_V-1$, where $\omega_V$ denotes the weight of the adjacent edges. We can thus part this contribution into two factors of $\sqrt{\omega_V-1}$ and shift those towards the adjacent edges.

By definition, if $q_k$ connects two $2$-valent vertices, $c_{w_k}$ equals $(w_k-1)\cdot w_k$ --- it obtains two factors of $\sqrt{w_k-1}$ from both its adjacent vertices, and it also contributes its  own weight, as every edge does. If $q_k$ connects a $2$-valent with a $3$-valent vertex, it obtains only one factor of $\sqrt{w_k-1}$. If both vertices of $q_k$ are $3$-valent, it obtains no such factor. It follows that if we reinterpret the  product $ \prod_V(\omega_V-1)\cdot \prod_e\omega_e$ as a product over edges by shifting the vertex contributions as square roots into both adjacent edges, we get exactly the contribution $c_{w_k}$ which is used to define the propagator function for the Feynman integral. 

Thus the multiplicity with which a labeled quotient cover contributes to $\tilde{h}_{\Gamma,\Omega,a}$
 exactly equals the contribution of its corresponding summand in the Feynman integral (up to the factor of $\frac{2^{g'}-\delta_{0c}}{2^{c+1}}  \cdot 2^{g-1}$
 ), and the equality holds.

\end{proof}

The following is the main theorem of this section and expresses the generation function of elliptic twisted Hurwitz numbers as a finite sum over Feynman integrals.

\begin{theorem}
\label{thm-feynman}
Fix a genus $g>2$. The generating series of twisted Hurwitz numbers can be expressed in terms of Feynman integrals:
$$\sum_d \tilde{h}_{d,g} q^d=2^{g-1}\cdot  \sum_{\Gamma} \frac{2^{\frac{1}{2}\cdot (g-c_\Gamma+1)}-\delta_{0c_\Gamma}}{2^{c_\Gamma+1}}\cdot \sharp \Aut(\Gamma) \sum_{\Omega} I_{\Gamma,\Omega}(q).$$
Here, the first sum on the right hand side goes over all Feynman graphs of genus $g$ and $c_{\Gamma}$ denotes their number of $2$-valent vertices, while the second sum goes over all orders $\Omega$.    

\end{theorem}

\begin{proof}
     For a fixed graph $\Gamma$, let $\tilde{h}_{d,\Gamma}$ be the number of (unlabeled) quotient covers of degree $d$ for which the combinatorial type of the source curve is $\Gamma$. As in Remark \ref{rem-countquotients}, each cover $\overline{\pi}$ is counted with multiplicity 
     $\frac{2^{g'}-\delta_{0c}}{2^{c+1}}  \cdot 2^{g-1}\cdot \frac{1}{|\Aut(\overline{\pi})|}\prod_V(\omega_V-1)\prod_{e}\omega(e)$.
 There exists a forgetful map $\ft$ from the set of labeled quotient covers to the set of unlabeled covers by just forgetting the labels. For an (unlabeled) quotient cover $\overline{\pi}$ whose source is of combinatorial type $\Gamma$, the automorphism group of $\Gamma$, $\Aut(\Gamma)$, acts transitively on the fiber $\ft^{-1}(\overline{\pi})$ by relabeling vertices and edges. So, to determine the cardinality of the set $\ft^{-1}(\overline{\pi})$, we think of it as the orbit under this action and obtain $\sharp \ft^{-1}(\overline{\pi})=\frac{\sharp\Aut(\Gamma)}{\sharp \Aut(\overline{\pi})}$, since the stabilizer of the action equals the set of automorphisms of $\overline{\pi}$.
Each labeled quotient  cover in the set $\ft^{-1}(\pi)$ is counted with the same multiplicity ${\frac{2^{g'}-\delta_{0c}}{2^{c+1}}  \cdot 2^{g-1}}
\cdot \prod_V(\omega_V-1)\prod_{e}\omega(e)$.

The sum $\sum_{a|\sum a_i=d} \sum_\Omega \tilde{h}_{\Gamma,\Omega,a}$
can be reorganized as a sum over unlabeled quotient covers, where for each unlabeled cover, we have to sum the multiplicities for each labeled quotient cover in the fiber under $\ft$. As the multiplicity is the same for each element in the fiber, and there are $\frac{\sharp\Aut(\Gamma)}{\sharp \Aut(\overline{\pi})}$ elements in the fiber, we can see that this sum equals 
$\sharp\Aut(\Gamma)\cdot \tilde{h}_{d,\Gamma}$.

We conclude 

\begin{align}  &\sum_d \tilde{h}_{d,g} q^{d} = \sum_d \sum_\Gamma \tilde{h}_{d,\Gamma} q^{d}= \sum_d \sum_\Gamma \frac{1}{\sharp\Aut(\Gamma)} \sum_{a|\sum a_i=d} \sum_\Omega\tilde{h}_{\Gamma,\Omega,a} q^{d} \end{align}

Now we can replace $\tilde{h}_{\Gamma,\Omega,a}$ by 
{$\frac{2^{g'}-\delta_{0c}}{2^{c+1}}  \cdot 2^{g-1}$}
times the coefficient of $q^{a}$ in $I_{\Gamma,\Omega}(q_1,\ldots,q_{r})$ by Proposition \ref{prop-hGammaOmega}.
If we insert $q_k=q$ for all $k$  we can conclude that the coefficient of $q^{d}$ in $I_{\Gamma,\Omega}(q)$ equals 
{$\frac{2^{g'}-\delta_{0c}}{2^{c+1}}  \cdot 2^{g-1}$} times $\sum_{a|\sum a_i=d}\tilde{h}_{\Gamma,\Omega,a}$. Thus the generating series above equals

$$\sum_d \tilde{h}_{d,g} q^d={2^{g-1}\cdot  \sum_{\Gamma} \frac{2^{\frac{1}{2}\cdot (g-c_\Gamma+1)}-\delta_{0c_\Gamma}}{2^{c_\Gamma+1}}\cdot \sharp \Aut(\Gamma)} \sum_{\Omega} I_{\Gamma,\Omega}(q).$$

\end{proof}

\section{The Fock space approach}
\label{sec-fockspace}

We shortly review the bosonic Fock space approach for generating series of Hurwitz numbers.

The bosonic Heisenberg algebra $\mathcal{H}$ is the Lie algebra with basis $\alpha_n$ for $n\in \mathbb{Z}$ such that for $n\neq 0$ the following commutator relations are satisfied:
\begin{align}
[\alpha_n, \alpha_m]=\left(n\cdot \delta_{n,-m}\right)\alpha_0,
\end{align}
where $\delta_{n,-m}$ is the Kronecker symbol and $[\alpha_n, \alpha_m]:=\alpha_n\alpha_m-\alpha_m\alpha_n$. 
The bosonic Fock space $F$ is a representation of $\mathcal{H}$.  It is generated by a single ``vacuum vector'' $v_\emptyset$.  The positive generators annihilate $v_\emptyset$: $\alpha_n\cdot v_\emptyset=0$ for $n>0$, $\alpha_0$ acts as the identity and the negative operators act freely.  That is, $F$ has a basis $b_\mu$ indexed by partitions, where
\begin{align}
b_\mu=\alpha_{-\mu_1}\dots \alpha_{-\mu_m}\cdot v_{\emptyset}.
\end{align}

We define an inner product on $F$ by declaring $\langle v_\emptyset | v_\emptyset \rangle=1$ and $\alpha_n$ to be the adjoint of $\alpha_{-n}$. 

We write 
$\langle v|A|w\rangle$ for $\langle v|Aw\rangle$, where $v,w\in F$ and the operator $A$ is a product of elements in $\mathcal{H}$, and $\langle A \rangle$ for $\langle v_\emptyset |A|v_\emptyset \rangle$. The first is called a \emph{matrix element}, the second a \emph{vacuum expectation}.
{We introduce a new formal variable $z$ to keep track of $4$-valent vertices, as their number influences the prefactor with which we have to count quotient covers by Remark \ref{rem-countquotients}, i.e. ${\frac{2^{g'}-\delta_{0c}}{2^{c+1}}  \cdot 2^{g-1}}$.}

\begin{definition} The vertex operator is defined by:
\begin{equation}
\label{caj}M  
=2\cdot \Big(\sum_{k>0} (k-1)\cdot \alpha_{-k}\alpha_k\cdot z
+
\frac{1}{2} \sum_{k>0} \sum_{\substack{0< i, j \\i+j=k}} \alpha_{-j}\alpha_{-i}\alpha_k+\alpha_{-k}\alpha_{i}\alpha_{j}\Big)
\end{equation}
\end{definition}

We note that unlike in the Feynman diagram approach here we don't have to shift vertex contributions into neighbouring edges. Moreover, the global factor of $2$ is to take the number of branch points into account. We can also view it as vertex contribution.

We obtain the following result.

\begin{proposition}\label{prop-doublehurwitzfock} The twisted double Hurwitz  number $\tilde{h}^\bullet_g(\mu,\nu)$ (see Definition 4, \cite{HM22}) equals a matrix element on the bosonic Fock space:
$${\tilde{h}^\bullet_g(\mu,\nu) = \frac{1}{\prod_i\mu_i\cdot \prod_j \nu_j}\sum_{c=0}^{g-1} \coef_{[z^c]}(\langle b_\mu | M^{g-1} | b_\nu\rangle)\cdot  \frac{2^{\frac{1}{2}(g-c+1)}}{2^{c+1}}  \cdot 2^{g-1}}.$$
\end{proposition}

\begin{proof}
    
This statement follows by combining Wick's Theorem with the Correspondence Theorem for twisted double Hurwitz numbers in \cite{HM22} resp.\ with the cut-and-join equation for twisted double Hurwitz numbers \cite[Theorem 6.5]{chapuy2020non}: Wick's Theorem (Theorem 5.4.3 \cite{CJMR16}, Proposition 5.2 \cite{BG14b}, \cite{Wic50}) expresses a matrix element as a weighted count of graphs that are obtained by completing local pictures. It turns out that the graphs in question are exactly the quotient covers we enumerate to obtain $\tilde{h}^\bullet_g(\mu,\nu)$, {multiplied with the factor depending on the number $c$ of $4$-valent vertices, as described in Remark \ref{rem-countquotients}.}

Notice that we have to use the disconnected theory here ($\bullet$), since the matrix element encodes \emph{all} graphs completing the local pictures and cannot distinguish connected and disconnected graphs.

The local pictures are built as follows: we draw one vertex for each vertex operator. For an $\alpha_n$ with $n>0$, we draw an edge germ of weight $n$ pointing to the right. If $n<0$, we draw an edge germ of weight $n$ pointing to the left.
For the two Fock space elements $b_\mu$ and $b_\nu$, we draw germs of ends: of weights $\mu_i$ on the left pointing to the right, of weights $\nu_i$ on the right pointing to the left. Wick's Theorem states that the matrix element $\langle b_\mu | M^n | b_\nu\rangle$ equals a sum of graphs completing all possible local pictures, where each graph contributes the product of the weights of all its edges (including the ends) and the vertex contributions arising from the vertex operator.
A completion of the local pictures can be interpreted as a quotient cover of $\mathbb{R}$ (with suitable metrization).

The vertex operator sums over all the possibilities of the local pictures for the graphs, i.e.\ it sums over all possibilities how a vertex of a quotient cover can look like. {The variable $z$ takes care of how many $4$-valent vertices there are.}
\end{proof}

Combining Proposition \ref{prop-doublehurwitzfock}
with the relation we obtain via cutting in the proof of Theorem \ref{thm-corres}, we can express twisted Hurwitz numbers of the elliptic curve in terms of matrix elements:

\begin{proposition}\label{prop-etop1}
A twisted Hurwitz number of the elliptic curve equals a weighted sum of twisted double Hurwitz numbers:
$$\tilde{h}^\bullet_{d,g}= \sum_{\mu\; \vdash d} \frac{\prod_i \mu_i}{|\Aut(\mu)|}
\tilde{h}^\bullet_g(\mu,\mu).
$$
Here, the sum goes over all partitions $\mu$ of $d$.
\end{proposition}

Proposition \ref{prop-etop1} is a corollary of the two Correspondence Theorems: given a tropical cover of $E$, let $\mu$ be the partition encoding the weights of the edges mapping to the base point $p_0$. We mark the preimages of $p_0$, for which we have $|\Aut(\mu)|$ choices. For each choice, we cut off $E$ at $p_0$ and the covering curve at the preimages of $p_0$, obtaining a cover of $\mathbb{R}$ with ramification profiles $\mu$ and $\mu$ above $\pm \infty$. The cut off tropical cover contributes to $\tilde{h}^\bullet_g(\mu,\nu)$, but its multiplicity differs from the multiplicity of the cover of $E$ by a factor of $\prod \mu_i$, since the edges we cut off are no longer bounded.

Finally, we obtain an expression of elliptic twisted Hurwitz numbers as a matrix element on the bosonic Fock space as a corollary of \cref{prop-doublehurwitzfock,prop-etop1}.

\begin{corollary}\label{cor-etop1}
A twisted Hurwitz number of the elliptic curve $E$ equals a sum of matrix elements on the bosonic Fock space:
$$\tilde{h}_{g,d}^\bullet = \sum_{\mu\;\vdash d}  \frac{1}{|\Aut(\mu)|\prod_i\mu_i} \sum_{c=0}^{g-1}\coef_{[z^c]}(\langle b_\mu | M^{g-1} | b_\mu\rangle)\cdot  \frac{2^{\frac{1}{2}(g-c+1)}}{2^{c+1}}  \cdot 2^{g-1}.$$
\end{corollary}

\printbibliography

\end{document}